\documentclass{aims}
\usepackage{amsmath}
  \usepackage{paralist}
  \usepackage{graphics} 
  \usepackage{epsfig} 
\usepackage{graphicx}  \usepackage{epstopdf}
 \usepackage[colorlinks=true]{hyperref}
\hypersetup{urlcolor=blue, citecolor=red}
\usepackage{hyperref}

\usepackage{latexsym}
\usepackage{array}
\usepackage{amssymb}
\usepackage{enumerate}
\usepackage{bbold}
\usepackage{color}
\usepackage{stmaryrd}
\usepackage{mathrsfs} 
\usepackage{caption}

\def\currentvolume{X}
 \def\currentissue{X}
  \def\currentyear{200X}
   \def\currentmonth{XX}
    \def\ppages{X--XX}
     \def\DOI{10.3934/xx.xx.xx.xx}

\newtheorem{theorem}{Theorem}[section]
\newtheorem{corollary}{Corollary}

\newtheorem{lemma}[theorem]{Lemma}
\newtheorem{proposition}{Proposition}

\theoremstyle{definition}
\newtheorem{definition}[theorem]{Definition}
\newtheorem{remark}{Remark}

\newtheorem{assumption}[theorem]{Assumption}
\newtheorem{example}[theorem]{Example}

\DeclareMathOperator{\Lap}{\mathcal{L}}

\DeclareMathOperator{\Four}{\mathscr{F}}

\title[Second order dispersion relations] 
      {Long-time behavior of second order linearized Vlasov-Poisson equations near a homogeneous equilibrium.}

\author[Joackim Bernier$^*$ and Michel Mehrenberger]{}


 \email{joackim.bernier@univ-rennes1.fr}
 \email{michel.mehrenberger@univ-amu.fr}

\thanks{This work was granted access to the HPC resources of Aix-Marseille Universit\'e financed by the project Equip@Meso (ANR-10-EQPX-29-01) of the program "œInvestissements d'Avenir" supervised by the Agence Nationale de la Recherche. This work has been carried out within the framework of the EUROfusion Consortium and has received funding from the Euratom research and training programme 2014-2018 and 2019-2020 under grant agreement No 633053. The views and opinions expressed herein do not necessarily reflect those of the European Commission.}

\thanks{$^*$ Corresponding author: Joackim Bernier}

\begin{document}
\maketitle

\centerline{\scshape Joackim Bernier$^*$}
\medskip
{\footnotesize
 \centerline{Univ Rennes, INRIA, CNRS, IRMAR - UMR 6625, F-35000 Rennes, France}
} 

\medskip

\centerline{\scshape Michel Mehrenberger}
\medskip
{\footnotesize
 \centerline{ Aix Marseille Univ, CNRS, Centrale Marseille, I2M, Marseille, France}
}

\bigskip

 \centerline{(Communicated by the associate editor name)}

\begin{abstract}
The asymptotic behavior of the solutions of the second order linearized Vlasov-Poisson system around homogeneous equilibria is derived. 
It provides a fine description of some nonlinear and multidimensional phenomena such as the existence of \emph{Best frequencies}. 
Numerical results for the $1D\times1D$ and $2D\times2D$ Vlasov-Poisson system illustrate the effectiveness of this approach.
\end{abstract}


\section{Introduction}
We consider 
potentials $\phi=\phi(t,x):\mathbb{R}\times{\mathbb{T}_d}  \to \mathbb{R}$ and
distribution functions $f=f(t,x,v):\mathbb{R}\times{\mathbb{T}_d}  \times \mathbb{R}^d \to \mathbb{R}$  satisfying the Vlasov-Poisson system
\begin{equation}
\label{VP_eq}
\tag{VP}
\left\{ \begin{array}{llll} 
\partial_t f + v\cdot\nabla_x f- \nabla_x \phi\cdot\nabla_v f =0 \\
\Delta_x \phi =n(f)- \int_{\mathbb{R}^d} f \textrm{d}v \\
f(t=0) = f_0.
\end{array} \right.
\end{equation}
Here periodic boundary conditions being used, 
${\mathbb{T}_d} $ is a $d$ dimensional torus: there exist $L_1, \dots, L_d>0$ such that
${\mathbb{T}_d} = (\mathbb{R}/L_1 \mathbb{Z})\times \dots \times (\mathbb{R}/L_d \mathbb{Z}).$
Furthermore, doing an assumption of neutrality, we only consider solutions of \eqref{VP_eq} such that
\[ n(f) = \frac{1}{L_1 \dots L_d}\iint_{{\mathbb{T}_d} \times \mathbb{R}^d} f \textrm{d}x \textrm{d}v.\]

In this paper, we aim at exhibiting nonlinear and multidimensional phenomena of solutions of \eqref{VP_eq}, pursuing a first preliminary work \cite{us} on this subject.
 Beyond their physical interest, these phenomena can be relevant to evaluate the performances and the qualitative properties of numerical methods.

 Since the very first 
 developments of numerical methods for solving \ref{VP_eq} (we refer to \cite{VlasovBookSonnen}, for a review; the literature is particularly huge in $1D\times 1D$ and we can mention \cite{shoucrigagne}, 
 as one of the earliest works in $2D\times 2D$),  
the numerical solutions are compared to the solutions of the Vlasov-Poisson system linearized around a homogeneous equilibria $f^{eq}\equiv f^{eq}(v)$\footnote{it can be noticed that every function depending only on $v$ is an equilibrium of \eqref{VP_eq}. }. It consists in looking for solutions of \eqref{VP_eq} of the type
$$
\left\{  \begin{array}{llllll} f &=& f^{eq} &+& \varepsilon g\\
								\phi &=& 0 &+& \varepsilon \psi. \end{array} \right.
$$
Neglecting second order terms, $g$ is formally a solution of
\begin{equation}
\label{VPL}
\tag{VPL}
\left\{ \begin{array}{llll}
\partial_t g + v\cdot\nabla_x g- \nabla_x \psi\cdot\nabla_v f^{eq} =0, \\
\displaystyle \Delta_x \psi + \int_{\mathbb{R}^d} g \ \textrm{d}v =0,\\
g(t=0) = g_0.
\end{array} \right.
\end{equation}
This equation being linear and homogeneous, it is natural to try to solve it realizing a Fourier transform we respect to the variable $x$. Thus, we get
\begin{equation}
\label{VPLF}
\tag{VPLF}
\left\{ \begin{array}{llll}
\partial_t \widehat{g} + i (v\cdot k)  \widehat{g}-  i \widehat{\psi}(k \cdot\nabla_v) f^{eq} =0, \\
\displaystyle -|k|^2  \widehat{\psi} + \int_{\mathbb{R}^d}  \widehat{g} \ \textrm{d}v =0,\\
 \widehat{g}(t=0) =  \widehat{g}_0,
\end{array} \right.
\end{equation}
where $k\in \widehat{\mathbb{T}}_d =  (2\pi/L_1) \mathbb{Z}\times \dots \times  (2\pi/L_d) \mathbb{Z}$ and the Fourier transform with respect to the space variable is defined for $u\in L^1({\mathbb{T}_d})$ and $k\in \widehat{\mathbb{T}}_d$ by
$$ \widehat{u}(k) =  \left(\prod_{j=1}^d L_j \right)^{-1} \int_{{\mathbb{T}_d}} u(x) e^{-ik\cdot x} \textrm{d}x. $$
It is relevant to notice on \eqref{VPLF} that there is no \emph{energy exchange} between space modes at the linear level.  In other words, if $g$ is a solution of \eqref{VPL} such that $g_0 \equiv\widehat{g_0}(v) e^{ik\cdot x}$ then it is of the form 
 $g(t,x,v) = \widehat{g}(t,v) e^{ik\cdot x}$. As a consequence, the linear analysis is not well suited to 
 describe multidimensional phenomena that could be confronted with numerical simulations.

Since \eqref{VPLF} is linear and autonomous, it is natural to solve it with the \emph{Laplace transform}. This transform is defined for functions $u:\mathbb{R}_+^{*}\to \mathbb{C}$ such that there exists $\lambda\in \mathbb{R}$ satisfying $ue^{-\lambda t}\in L^{\infty}(\mathbb{R}_+^*)$ and $z\in \mathbb{C}$ such that $\Im z>\lambda$ by
\[\Lap [u](z)  =  \int_{0}^{\infty} u(t) e^{izt} \textrm{d}t.  \]
Thus, it can be proven that solutions of \eqref{VPLF} are given by
\begin{equation}
\label{intro_art_duham}
g(t,x,v) = \sum_{k \in \widehat{{\mathbb{T}_d}}} e^{ik\cdot(x-vt)} \widehat{g_0}(k,v) + i \int_{0}^t e^{i k\cdot(x-v(t-s))} \widehat{\psi}(s,k) k\cdot\nabla_v f^{eq}(v) \textrm{d}s,
\end{equation}
and for $\Im z$ large enough
\begin{equation}
\label{intro_art_rel_disp_lin}
\Lap \left[ \widehat{\psi}(t,k)\right] (z) = \frac{N_k(z)}{D_k(z)} =:M_k(z).
\end{equation}
where $N_k$ and $D_k$ are holomorphic functions defined when $\Im z$ is large enough by
\begin{equation}
\label{intro_defNkDk}
N_k(z) = -\frac{i}{|k|^2}  \int_{\mathbb{R}^d} \frac{\widehat{g_0}(k,v)}{v\cdot k-z}  \textrm{d}v \quad  \textrm{ and }  \quad D_k(z) = 1- \frac1{|k|^2} \int \frac{k\cdot \nabla_v f^{eq}(v)}{v\cdot k-z} \textrm{d}v.
\end{equation}
Thus to get a solution $g$ of \eqref{VPL} by \eqref{intro_art_duham} we just have to solve the equation \eqref{intro_art_rel_disp_lin} (called \emph{dispersion relation}) determining an inverse Laplace transform. 

Up to some strong assumptions on $f^{eq}$ and $\widehat{g_0}(k)$ (precised later), it can be proven that $N_k$ and $D_k$ are entire functions and that, for all $\lambda\in \mathbb{R}$, the number of zeros of $D_k$ with an imaginary part larger than $\lambda$ is finite (see Remark \ref{remark_NkDk}). Thus, using the formula
\begin{equation}
\label{intro_art_pas_fou}
\Lap[t^m e^{-i\omega t}](z) =\frac{i^{m+1} m!}{(z-\omega)^{m +1}}, \ \omega \in \mathbb{C}, \ m\in \mathbb{N},
\end{equation}
and realizing precise estimates of remainder terms, we can prove that \eqref{intro_art_rel_disp_lin} has an analytic solution $\widehat{\psi}$ whose analytic expansion is given, for all $\lambda\in \mathbb{R}$, by
\begin{equation}
\label{intro_art_DA_psi}
\widehat{\psi}(t,k) = \sum_{\substack{D_k(\omega)=0\\ \Im \omega\geq \lambda}} P_{\omega,k}(t) e^{-i\omega t} + \mathcal{O}(e^{\lambda t}), 
\end{equation}
where $P_{\omega,k}$ is the polynomial such that $\displaystyle M_k(z) \mathop{=}_{z\to \omega}  \Lap [P_{\omega,k}(t) e^{-i\omega t}] (z) + \mathcal{O}(1)$ is the expansion of $M_k(z)$ in $\omega$.

Such an analysis was first realized by Landau \cite{Landau46}, in 1946. It has been done rigorously and generalized in 1986 by Degond \cite{MR825714}. It gave a partial explanation to the phenomenon of \emph{Landau damping}. This latter corresponds to the dynamic of \eqref{VP_eq} when for all $k\in \widehat{{\mathbb{T}_d}}$, $D_k(z)$ does not vanish if $\Im z \geq 0$. In this case, the electric potential goes exponentially fast to zero as $t$ goes to $+\infty$. In 2011, Mouhot and Villani proved the existence of this phenomenon for the nonlinear Vlasov-Poisson equation \eqref{VP_eq} in  \cite{MouhotVillani}.

As we have just seen, due to the absence of energy exchange between the spaces modes at the linear level, the linearization is not relevant to explain really multidimensional phenomena. Furthermore, of course, it can not explain nonlinear phenomena. 
This motivates thus the study of  the dynamic of the \emph{second order term} in the expansion of $f$ as powers of $\varepsilon$. More precisely, we look for a solution of \eqref{VP_eq} under the form
$$ \left\{ \begin{array}{lllllllll} f & = & f^{eq} & +  & \varepsilon g & + & \varepsilon^2 h & + &  o(\varepsilon^2), \\
								  \phi & = & 0 & + & \varepsilon \psi & + & \varepsilon^2 \mu & +  &o(\varepsilon^2),
\end{array} \right. $$
where $h(t=0) \equiv 0$. Neglecting the third order terms, it can be proven formally that $(h,v)$ is a solution of
\begin{equation}
\label{VPL2}
\tag{VPL2}
\left\{ \begin{array}{llll}
\partial_t h + v\cdot\nabla_x h- \nabla_x \mu\cdot\nabla_v f^{eq} =  \nabla_x \psi\cdot\nabla_v g, \\
\displaystyle \Delta_x \mu + \int_{\mathbb{R}^d} h \ \textrm{d}v =0,\\
h(t=0) = 0.
\end{array} \right.
\end{equation}
We recognize the linearized Vlasov-Poisson equations, with an initial condition equal to zero but with a source term. 
In that case, we refer to Denavit \cite{denavit}, for one of the first works on the subject, in 1965. Different second order oscillations appear and have been studied by physicists (see for example \cite{O2_phy} and references therein; 
there are many references especially from the 1960s and 1970s).  Our aim is to make here a rigorous mathematical study of the asymptotical behavior of the solutions of these equations, which has, to the best of our knowledge, 
not already been performed.

 This expansion in power of $\varepsilon$ is quite natural because if we assume that $f$, the solution of \eqref{VP_eq}, is a smooth function of $\varepsilon$ then $g$ and $h$ are just its first and second Taylor coefficients around $\varepsilon=0$. We admit that there is a priori no reason that this second order approximation provides a more accurate approximation of $f$ for long times than the usual first order approximation. Nevertheless, numerical results suggest that this second order approximation is relevant for long times both if the equilibrium is stable or not (see \cite{us}, \cite{O2_phy} and Section \ref{sec:5}). Furthermore, to the best of our knowledge, the asymptotic expansion of the solutions of \eqref{VPL2} we realize in this paper (see Theorem \ref{The_th}) is much more precise than what is known for the asymptotic behavior of the solution of \eqref{VP_eq} (see \cite{MR3489904},\cite{MouhotVillani}).

%

Here, as for the linear case, we solve \eqref{VPL2} using a Laplace transform for the time variable and a Fourier transform for the space variable. More precisely, some calculations prove that  \eqref{VPL2} is equivalent to 
\begin{multline*}
\label{intro_art_g_Duhamel_2}
h(t,x,v) = \sum_{k \in \widehat{{\mathbb{T}_d}}}  i \int_{0}^t e^{i k\cdot(x-v(t-s))} \widehat{\mu}(s,k) k\cdot\nabla_v f^{eq}(v) \textrm{d}s \\
+ \int_{0}^t \widehat{\nabla_x \psi\cdot\nabla_v g}(s,k,v) e^{i k\cdot(x-v(t-s))}   \textrm{d}s ,
\end{multline*}
and when $\Im z$ is large enough 
\begin{equation}
\label{intro_art_rel_disp_lin2}
\Lap \left[ \widehat{\mu}(t,k)\right] (z) = \frac{  \mathcal{N}_{k}(z)}{D_k(z)} =:\mathcal{M}_k(z),
\end{equation}
where $D_k$ is given by \eqref{intro_defNkDk} and $\mathcal{N}_k$ is a meromorphic function on $\mathbb{C}$ explicitly known.

As previously, there is just to invert a Laplace transform to solve \eqref{VPL2}. As for the linear case, a precise study of $\mathcal{M}_k$ and its poles gives a solution $\mu$ of \eqref{intro_art_rel_disp_lin2} and an asymptotic expansion of the form 
\begin{equation}
\label{intro_art_DA_mu}
\forall \lambda\in \mathbb{R}, \ \widehat{\mu}(t,k) =  \sum_{\substack{\mathcal{M}_k(\omega)=\infty\\ \Im \omega\geq \lambda}} Q_{\omega,k}(t) e^{-i\omega t} + \mathcal{O}(e^{\lambda t}).
\end{equation}
where $Q_{\omega,k}$ is the polynomial such that $\displaystyle \mathcal{M}_k(z) \mathop{=}_{z\to \omega}  \Lap [Q_{\omega,k}(t) e^{-i\omega t}] (z) + \mathcal{O}(1)$.

The poles of $\mathcal{M}_k$ are of two kinds: they can be zeros of $D_k$ (generating the same frequencies as at the first order) or poles of $\mathcal{N}_{k}$. The study of the poles is technical because $\mathcal{N}_{k}$ is defined from the solution of \eqref{VPL}. However, the asymptotic expansion of $\psi$ (see \eqref{intro_art_DA_psi}) enables a decomposition of
$\mathcal{N}_{k}$ in more elementary terms whose poles can be determined.
 
In order to give an intuition of these poles, we consider a term that is very representative\footnote{but slightly simplified.} of this decomposition:
 $$
\mathcal{N}_{k}^{(rep)}(z)=   \Lap \left[ F^{(rep)}_k(t) \right] (z)
$$
where
\begin{equation}
\label{intro_Frep}
F^{(rep)}_k(t) =    e^{-i (\omega_1 t + \omega_2 t) }  \iint_{0\leq \tau\leq s\leq t}  e^{i (\omega_1 \tau + \omega_2 s) }   \mathscr{F}[f^{eq}] ( \tau k_1 + s k_2) \ \textrm{d}s \ {\rm d\tau}
\end{equation}
with $k_1,k_2 \in \widehat{{\mathbb{T}_d}}\setminus{0}$ satisfy $k_1+k_2 = k$, $\omega_1,\omega_2 \in \mathbb{C}$ are such that $D_{k_1}(\omega_1) = D_{k_2}(\omega_2) = 0$  and $ \mathscr{F}[f^{eq}]$ is the Fourier transform of $f^{eq}$. The later being defined for $u\in L^1(\mathbb{R}^d)$ and $\xi \in \mathbb{R}^d$ by
$$
 \Four [u](\xi)  = \int_{\mathbb{R}^d} u(v) e^{-iv\cdot \xi} \textrm{d}v.  
$$

Since $\mathcal{N}_{k}^{(rep)}$ is the Laplace transform of $F^{(rep)}_k(t)$, it can be proven that its poles are given by the asymptotic expansion of $F^{(rep)}_k(t)$ with the formula \eqref{intro_art_pas_fou}. As it is suggested by the formula \eqref{intro_Frep}, the behavior of this later is quite different if the set of the points $(\tau,s)$ such that $\tau k_1 + s k_2=0$ is a line segment (\emph{resonant case}) or a point (\emph{non-resonant case}).

In the non-resonant case, there exists a constant $c>0$ such that
$$0\leq \tau\leq s\leq t, \ |\tau k_1 + s k_2| \geq c s.$$ 
So, assuming that $f^{eq}$ is regular enough so that $\mathscr{F}[f^{eq}](\xi)$ decreases faster than any exponential as $|\xi|$ goes to $+\infty$ (for example like a Gaussian), we can prove that the integral in \eqref{intro_Frep} converges faster than any exponential as $t$ goes to $+\infty$. As a consequence, we get a constant $a\in \mathbb{C}$ such that 
$$
\forall \lambda\in \mathbb{R}, \ F^{(rep)}_k(t) = a e^{-i (\omega_1 t + \omega_2 t) } + \mathcal{O}(e^{\lambda t}).
$$ 

In the resonant case, there exists $\gamma\in (0,1)$ such that
$$  k_2 = - \gamma k_1.$$
Realizing a natural change of coordinates in \eqref{intro_Frep}, we get 
$$
F^{(rep)}_k(t) =    \int_0^t \int_{-\gamma s}^{(1-\gamma)s}   e^{-i (\omega_1 (t-\tau -\gamma s)  + \omega_2 (t-s)  ) }   \mathscr{F}[f^{eq}] ( \tau k_1 )  \ \textrm{d}\tau \ \textrm{d}s.
$$
Thus, assuming that $f^{eq}$ is regular enough so that $\mathscr{F}[f^{eq}](\xi)$ decrease faster than any exponential as $|\xi|$ goes to $+\infty$, we have
\begin{equation*}
\begin{split}
F^{(rep)}_k(t)  =&  \left(\int_0^t  e^{-i (\omega_1 (t -\gamma s)  + \omega_2 (t-s)  ) } \textrm{d}s\right) \left( \int_{\mathbb{R}}  e^{i\omega_1 \tau} \mathscr{F}[f^{eq}] ( \tau k_1 ) \textrm{d}\tau   \right)\\
&-  e^{-it(\omega_1+\omega_2)} \int_0^\infty \int_{\substack{ \tau\geq (1-\gamma)s \\ \textrm{or } \tau<-\gamma s}} e^{i (\omega_1 (\tau +\gamma s)  + \omega_2 s  ) }   \mathscr{F}[f^{eq}] ( \tau k_1 )  \ \textrm{d}\tau \ \textrm{d}s \\
&+e^{-it(\omega_1+\omega_2)} \int_t^\infty \int_{\substack{ \tau\geq (1-\gamma)s \\ \textrm{or } \tau<-\gamma s}} e^{i (\omega_1 (\tau +\gamma s)  + \omega_2 s  ) }   \mathscr{F}[f^{eq}] ( \tau k_1 )  \ \textrm{d}\tau \ \textrm{d}s,
\end{split}
\end{equation*}
and we can prove that the third term decreases faster than any exponential. Thus, this decomposition provides the following asymptotic expansion 
$$
\forall \lambda\in \mathbb{R}, \ F^{(rep)}_k(t) = a e^{-i t (\omega_1+\omega_2)} + b e^{-i t \omega_b} + \mathcal{O}(e^{\lambda t}),
$$ 
where $a,b\in \mathbb{C}$ and $\omega_b =  (1-\gamma) \omega_1  = (|k|/|k_1|) \omega_1$ is the \emph{Best frequency} (according to \cite{O2_phy}).

\medskip

As suggested by this sketch of proof, we can prove that $\mathcal{M}_k$ have three kinds of poles. More precisely, if $\omega$ is a pole of $\mathcal{M}_k$ it satisfies one of the following conditions
\begin{enumerate}[(I)]
\item $\omega$ is a zero of $D_k$,
\item $\omega = \omega_1+\omega_2$ where $D_{k_1}(\omega_1) = D_{k_2}(\omega_2)=0$ and $k_1+k_2=k$,
\item $\omega = (|k|/|k_1|) \omega_1$ where $D_{k_1}(\omega_1) =0$ and there exists $\gamma\in (0,1)$ such that $k = \gamma k_1$.
\end{enumerate}
We recall that these poles drive the asymptotic behavior of $\widehat{\mu}(k)$ through formula \eqref{intro_art_DA_mu}.
The frequencies (I) and (II) have already been identified in our preliminary work on this subject \cite{us}, but not the frequency (III). We emphasize that all the three type of frequencies are listed in \cite{O2_phy}, 
which makes our analysis coherent with the physics litterature.

\medskip

To conclude this presentation, we are going to state a precise theorem giving the asymptotic behavior of the solutions of \eqref{VPL2}. To this end, we need to introduce some notations.

\begin{definition}
Let $\mathscr{E}(\mathbb{R}^d)$ be the subspace of the Schwartz space $\mathscr{S}(\mathbb{R}^d)$, of functions $f$, whose Fourier transform, $\Four f$, extends to an entire function on $\mathbb{C}^d$ and such that 
\begin{equation}
\label{def_Erd} \exists \alpha\in (0,\frac{\pi}2), \forall \beta \in (0,\alpha),\forall \lambda\in \mathbb{R}, \ \sup_{x\in \mathbb{R}^d} \sup_{\theta \in (-\beta,\beta)} e^{\lambda |x|}|\Four f(e^{i\theta}x)| 
<\infty,
\end{equation}
where $|\cdot|$ denotes the canonical Hermitian norm of $\mathbb{C}^d$.
\end{definition}

\begin{remark} Most of our results require that $f^{eq} \in \mathscr{E}(\mathbb{R}^d)$ and $v\mapsto \widehat{g_0}(k,v) \in \mathscr{E}(\mathbb{R}^d)$. This assumption is probably not optimal but it is crucial in our proof in order to invert easily some Laplace transforms (see Theorem \ref{theo_anal_rep} and Lemma \ref{Hilb_to_Four}). Furthermore the space $\mathscr{E}(\mathbb{R}^d)$ contains most of the usual functions used in Vlasov-Poisson simulations. For example, the Maxwellian functions belong to this space. Appendix \ref{Appendix1} provides many examples and details about this space. 
\end{remark}

\begin{remark}
\label{remark_NkDk} Assuming that $f^{eq} \in \mathscr{E}(\mathbb{R}^d)$ and $v\mapsto \widehat{g_0}(k,v) \in \mathscr{E}(\mathbb{R}^d)$, $D_k$ and $N_k$ are entire functions and for all $\lambda\in \mathbb{R}$, the number of zeros of $D_k$ with an imaginary part larger than $\lambda$ is finite (proof will be given in Corollary \ref{cor_Dk_cool} and Proposition \ref{Nk_nice}). Appendix \ref{App:Computation of the zeros} provides an algorithm to computate the zeros of $D_k$.
\end{remark}

\begin{definition} If $k\in \widehat{{\mathbb{T}_d}}$, $n_{k,\omega}$ denotes the multiplicity of $\omega$ as zero of $D_k$, i.e.
$$
 n_{k,\omega} = \max \{ m \in \mathbb{N} \ | \ \forall \ell < m, \ D_k^{(\ell)}(\omega)=0  \}. 
 $$
\end{definition}

Most of the result of this paper will require that $g_0$ is supported on a finite number of spatial modes whose set is denoted $K\subset \widehat{{\mathbb{T}_d}} \setminus \{ 0 \}$. More precisely, they require the following assumption
 \begin{assumption} $\empty$
 \label{assump}
 There exists $K$, a finite part of $\widehat{{\mathbb{T}_d}} \setminus \{ 0 \}$ such that 
\[ \forall x\in{\mathbb{T}_d}, \ g_0(x,v) = \sum_{k\in K} e^{ik\cdot x} \widehat{g_0}(k,v), \ \textrm{ with } v\mapsto\widehat{g_0}(k,v)\in \mathscr{E}(\mathbb{R}^d).\]
 \end{assumption}
 This assumption seems clearly not optimal but it is general enough to exhibit the relevant phenomena and it corresponds to the usual initial data used for numerical simulations. Furthermore, it simplifies most of the proof avoiding several problems of convergences.
 
\medskip

We can now state the main result of this paper: the following theorem proves the existence of smooth solutions of \eqref{VPL} and \eqref{VPL2} and describes their asymptotic behavior.
\begin{theorem} 
\label{The_th}
Let  $f^{eq} \in \mathscr{E}(\mathbb{R}^d)$ and $g_0$ be a function satisfying Assumption \ref{assump}. Then there exist  two $C^{\infty}$ functions $\psi,\mu:\mathbb{R}_{+}^*\times{\mathbb{T}_d} \to \mathbb{R}$ and two continuous functions $g,h:\mathbb{R}_{+} \times{\mathbb{T}_d} \times \mathbb{R}^d\to \mathbb{R}$, $C^{\infty}$ on $\mathbb{R}_{+}^* \times{\mathbb{T}_d} \times \mathbb{R}^d$, such that $(g,\psi,h,\mu)$ is solution of \eqref{VPL} and \eqref{VPL2}.

Furthermore, if $\lambda\in \mathbb{R}$, $\psi$ is a linear combination of functions of the two types
$$
J(t,x) =  t^{m}e^{i(k\cdot x-\omega t)} \textrm{ and } R(t,x) = r(t) e^{i(k\cdot x-i\lambda t)}
$$
where $k\in K$, $D_k(\omega) = 0$, $\Im \omega \geq \lambda$ and $0\leq m<n_{k,\omega}$ and $r$ is a bounded analytic function on $\mathbb{R}_+^*$. 

Similarly, $\mu$  is a linear combination of functions of the four types
$$
\begin{array}{lllllll}
J(t,x) &=&  t^{m}e^{i(k\cdot x-\omega t)} & & I(t,x) &=&  t^{\ell}e^{i(k\cdot x-(\omega_1+\omega_2) t)}  \\
B(t,x) &=&  t^{p}e^{i(k\cdot x-\frac{|k|}{|k_1|}\omega_1 t)}d_{k_1}^{k_2} & & R(t,x) &=& r(t) e^{i(k\cdot x-i\lambda t)}  \\
\end{array}
$$
where $k=k_1+k_2$, $r$ is a bounded analytic function on $\mathbb{R}_+^*$ and $k_1,k_2 \in K$ satisfy
$$
\left\{ \begin{array}{lll}
D_k(\omega) = D_{k_1}(\omega_1) = D_{k_2}(\omega_2)=0\\
k\cdot k_1 \neq 0 \textrm{ and } \left(d_{k_1}^{k_2} \neq 0 \iff \exists \gamma\in (0,1), \ k = \gamma k_1\right) \\
\Im \omega \geq \lambda \textrm{ and } \left( \Im \omega_1+\Im \omega_2\geq \lambda \textrm{ or } \frac{|k|}{|k_1|}\Im \omega_1\geq \lambda\right) \\
m < n_{k,\omega}, \ \ell < n_{k_1,\omega_1}+ n_{k_2,\omega_2}-1 + \sigma_{\omega_1,\omega_2}^{k_1,k_2}, \ p < n_{k_1,\omega_1} + 1+\nu_{\omega_1,\omega_2}^{k_1,k_2},
\end{array} \right.
$$
with $\sigma_{\omega_1,\omega_2}^{k_1,k_2}, \nu_{\omega_1,\omega_2}^{k_1,k_2}$ some non negative integers equal to zero in the non degenerate cases (see Remark \ref{relou} for details).  
\end{theorem}

\begin{remark} 
We have
$e^{-i\omega t} = e^{\Im(\omega)t-i\Re(\omega )t}$, 
so that all the terms of the sum except maybe the remainder $R$ are of the form $e^{ik\cdot x}P(t)e^{\tilde{\lambda}t+i\beta t}$, with $P(t)\in \mathbb{C}[t]$. If one of this term satisfies $\tilde{\lambda}<\lambda$, it can be put in the remainder term $R$.


\end{remark}
\begin{remark}
Taking $\lambda$ decreasing to $-\infty$ makes the sum larger, but it always remain finite, for a fixed $\lambda$,  since $K, K+K$ are finite together with the zeros (see Remark \ref{remark_NkDk}).  We warn the reader that, a priori, the expansion does not converge as $\lambda$  goes  to $-\infty$.
\end{remark}

\begin{remark}
\label{relou}
It may exist some degenerate cases for which the four types of functions introduced in the second part of Theorem \ref{The_th} are non distinct. In such a case, the numbers $\sigma_{\omega_1,\omega_2}^{k_1,k_2}$ and $\nu_{\omega_1,\omega_2}^{k_1,k_2}$ do not vanish and we have
$$\sigma_{\omega_1,\omega_2}^{k_1,k_2} = (n_{k,\omega_1+\omega_2}-1) \mathbb{1}_{D_k(\omega_1+\omega_2)=0} +2 \cdot \mathbb{1}_{d_{k_1}^{k_2}\neq 0 \textrm{ and } \omega_1+\omega_2 = \frac{|k|}{|k_1|}\omega_1}$$
 and  $$\nu_{\omega_1,\omega_2}^{k_1,k_2} = (n_{k,\frac{|k|}{|k_1|}\omega_1}-1) \mathbb{1}_{D_k(\frac{|k|}{|k_1|}\omega_1)=0}$$
 where $\mathbb{1}_{P}$ denotes the characteristic function of the property $P$.
\end{remark}
 
 \begin{remark} 
 In the case where $d_{k_1}^{k_2}\neq0$, which we will call \emph{resonant case}, where the \emph{Best frequency}, that is the term $B$ appears, 
 $p$ can a priori be $\ge 1$. For the $J$ and $I$ terms, the multiplicity can be equal to one, corresponding to $m=\ell=0$.
 \end{remark}

  \begin{remark} It is quite direct to extend the classical linear analysis of the Vlasov Poisson with $1$ specie to the case of multi-species charged particles (see $\mathsection$ 3.1.2 in \cite{us}). Similarly, we expect that it would be possible to extend our second order analysis to the multi-species case. Actually, such an extension is discussed in our preliminary work ($\mathsection$ 5.1 in \cite{us}) to explain more precisely the numerical results associated with a multi-species test case introduced in \cite{MR3517446}.
   \end{remark}
   
 \begin{remark}
   It may be interesting to try to extend this second order analysis to non-homogeneous equilibria. However such an analysis seems much more involved. Indeed, in this case, to carry out an analysis of the linearized equation similar to the analysis of the homogeneous case, it is usual to use action-angle variables (see for example \cite{MR3064331},\cite{horsin}). However, this change of coordinates generates some singularities and boundary effects leading to an algebraic decay of the electric field.
 \end{remark}

 
The fifth section of this paper is devoted to some numerical experiments. They principally aim at highlighting the Best's waves because most of the other phenomena associated with second order terms have been studied numerically in the proceedings \cite{us}. Unlike the linear case, it seems that there is no elementary way to determine \emph{a priori} the coefficients associated with the asymptotic expansion of $\mu$. Indeed, they depend non trivially on the solution of \eqref{VPL} (and not only on its asymptotic expansion). Consequently, we use here least squares procedures, which permit to have a simple and quick way to find these coefficients. 

There are some difficulty arising of these computations because, as we compare the solution of the second order expansion to the solution of \eqref{VP_eq}, this gives a constraint on $\varepsilon$ and the final time that should be small enough. As we have seen, the final time should also not be too small in order to be in the asymptotic regime, and this is also true for $\varepsilon$ (which is here put to the square, as we consider second order expansion) due to the limits imposed by machine precision. 

We admit that for the numerical checking of codes, second order terms have not gained much popularity, 
maybe as the linear terms generally already give  the  main phenomena. We emphasize that we are here able to identify the contributions of the different frequencies, 
and thus do an effective comparison with, as already told, multidimensional and nonlinear features.

\medskip

 \paragraph{\emph{Some remarks about the notations}} In order to keep proofs as readable as possible, we do some classical abuses of notation for integral transforms. For example, the Fourier transform on $\mathbb{R}^d$ is always associated with the variable $v$, it means that if $u\in L^1(\mathbb{R}^d)$ then $\Four[u]$ and $\Four[u(v)]$ denotes the same functions. Similarly, if $u$ is a function of $t,x,v$ then $\Four[u](t,x,\xi)$ denotes $\Four[v\mapsto u(t,x,v)](t,x,\xi)$. Similarly, $t$ is associated with $\Lap$, $z$ with $\Lap^{-1}$, $\xi$ with $\Four^{-1}$, $x$ with $u\mapsto \widehat{u}$ and $k$ with $(u\mapsto \widehat{u})^{-1}$.
 
 \medskip
 
 \paragraph{\emph{Outline of the work}}
 In Section \ref{sec:2}, we derive some integral equations (called dispersion relations) satisfied by solutions $\psi,\mu$ of \eqref{VPL} and \eqref{VPL2}. 
 Then we prove that it is enough to solve these dispersion relations to get solutions for \eqref{VPL} and \eqref{VPL2}. 
 The next two sections are devoted to the resolution of these dispersion relations and to the asymptotic expansions of their solutions: Section \ref{section_linear_resolution} is for the first order expansion and Section \ref{sec:4} is for the second order expansion. 
 Finally in Section \ref{sec:5}, we give some numerical results.




\section{Derivation of the dispersion relations}
\label{sec:2}

\subsection{Dispersion relations for first and second order} 

In the following propositions, we give the \emph{dispersion relations}, that are obtained through Fourier and Laplace transforms.
Note that we have an expression for both the electric potentials $\psi$ resp. $\mu$ and the distribution function $g$ resp. $h$ of the first  resp.  second order dispersion relations. 

\begin{proposition}
\label{propo_existence_lin}
Assume $f^{eq}\in \mathscr{E}(\mathbb{R}^d)$ and $g_0$ satisfies Assumption \ref{assump}. Assume there exists a $C^{\infty}$ function on $\mathbb{R}_+^*\times{\mathbb{T}_d}$, denoted $\psi$, and there exists $\lambda_0 >0$ such that $e^{ -\lambda_0 t}\psi(t,x)$ is bounded on $\mathbb{R}_+^*\times{\mathbb{T}_d}$. Furthermore, assume that, for all $k \in \widehat{{\mathbb{T}_d}}\setminus \{0\}$, $\Lap \left[ \widehat{\psi}(t,k)\right] (z) $ is a solution of 
\begin{equation}
\label{rel_disp_lin}
\Lap \left[ \widehat{\psi}(t,k)\right] (z) D_k(z) = -\frac{i}{|k|^2}  \int_{\mathbb{R}^d} \frac{\widehat{g_0}(k,v)}{v\cdot k-z}  \textrm{d}v.
\end{equation}
for $\Im z > \lambda_0$. \\
If we define $g$ by
\begin{equation}
\label{g_Duhamel_1}
g(t,x,v) = \sum_{k \in K} e^{ik\cdot(x-vt)} \widehat{g_0}(k,v) + i \int_{0}^t e^{i k\cdot(x-v(t-s))} \widehat{\psi}(s,k) k\cdot\nabla_v f^{eq}(v) \textrm{d}s,
\end{equation}
then $g$ is a $C^{\infty}$ function on $\mathbb{R}_+^*\times{\mathbb{T}_d}\times \mathbb{R}^d$, continuous on $\mathbb{R}_+\times{\mathbb{T}_d}\times \mathbb{R}^d$ and $(g,\psi)$ is a solution of \eqref{VPL}.
\end{proposition}

\begin{proposition}
\label{propo_existence_o2}
Assume $f^{eq}\in \mathscr{E}(\mathbb{R}^d)$ and $g_0$ satisfies Assumption \ref{assump}. Assume there exists a solution of \eqref{VPL} as in Proposition \eqref{propo_existence_lin}.
Assume there exists a $C^{\infty}$ function on $\mathbb{R}_+^*\times{\mathbb{T}_d}$, denoted $\mu$, and there exists $\lambda_1 >2\lambda_0$ such that $e^{ -\lambda_1 t}\psi(t)$ is bounded on $\mathbb{R}_+^*\times{\mathbb{T}_d}$. Furthermore, assume that, for all $k \in \widehat{{\mathbb{T}_d}}\setminus \{0\}$, $\Lap \left[ \widehat{\mu}(t,k)\right] (z) $ is a solution of 
\begin{equation}
\label{rel_disp_o2}
\Lap \left[ \widehat{\mu}(t,k)\right] (z) D_k(z) = -\frac{i}{|k|^2}  \int_{\mathbb{R}^d} \frac{\Lap \left[ \widehat{\nabla_x \psi\cdot\nabla_v g}(t,k,v)   \right](z)}{v\cdot k-z}  \textrm{d}v.
\end{equation}
for $\Im z > \lambda_1$. \\
If we define $h$ by
\begin{multline*}
\label{g_Duhamel_2}
h(t,x,v) = \sum_{k \in K+K}  i \int_{0}^t e^{i k\cdot(x-v(t-s))} \widehat{\mu}(s,k) k\cdot\nabla_v f^{eq}(v) \textrm{d}s \\
+ \int_{0}^t \widehat{\nabla_x \psi\cdot\nabla_v g}(s,k,v) e^{i k\cdot(x-v(t-s))}   \textrm{d}s ,
\end{multline*}
then $h$ is a $C^{\infty}$ function on $\mathbb{R}_+^*\times{\mathbb{T}_d}\times \mathbb{R}^d$, continuous on $\mathbb{R}_+\times{\mathbb{T}_d}\times \mathbb{R}^d$ and $(h,\mu)$ is a solution of \eqref{VPL2}.
\end{proposition}

\subsection{A general linearized Vlasov-Poisson equation} In order to prove Propositions \ref{propo_existence_lin} and \ref{propo_existence_o2}, 
as  \eqref{VPL} and  \eqref{VPL2} share the same structure, 
we focus on a general linearized Vlasov Poisson equation
\begin{equation}
\label{VP_lin_gen}
\tag{VPLG}
\left\{ \begin{array}{llll}
\partial_t \mathfrak{g}(t,x,v) + v\cdot\nabla_x \mathfrak{g}(t,x,v) - \nabla_x \mathfrak{u}(t,x)\cdot \nabla_v f^{eq}(v) = \mathfrak{S}(t,x,v), \\
\Delta_x  \mathfrak{u}(t,x) + \int \mathfrak{g}(t,x,v) \textrm{d}v =0, \\
\mathfrak{g}(0,x,v) = \mathfrak{g}_0(x,v).
\end{array} \right.
\end{equation}
In the following proposition, we derive a general dispersion relation satisfied by $ \mathfrak{u}$. We first do not consider the coupling with the Poisson equation.
\begin{proposition}
\label{prop_anal_lin} Assume $\mathfrak{g}_0 \in C^1({\mathbb{T}_d}\times \mathbb{R}^d)$, $f^{eq}\in C^2(\mathbb{R}^d)$ and $\mathfrak{S}(t,x,v)\in C^1(\mathbb{R}_+^*\times{\mathbb{T}_d}\times \mathbb{R}^d)$  and there exist $\lambda_0>0$, $\mathfrak{d} \in C^0(\mathbb{R}^d)\cap L^1(\mathbb{R}^d)$ satisfying
\[ \forall (t,k,v) \in \mathbb{R}_+^*\times \widehat{{\mathbb{T}_d}} \times \mathbb{R}^d, \ e^{-\lambda_0 t} |\widehat{\mathfrak{S}}(t,k,v)|+|\widehat{\mathfrak{g}_0}(k,v)| + |\nabla_v f^{eq}(v)| \leq \mathfrak{d}(v).\]
Assume there exists $ \mathfrak{u} \in C^1(\mathbb{R}_+^*\times{\mathbb{T}_d})$ such that $e^{ -\lambda_0 t}\mathfrak{u}(t)$ is bounded on $\mathbb{R}_+^*\times{\mathbb{T}_d}$. 
 Assume there exists a continuous function $\mathfrak{g}\in  C^1(\mathbb{R}_{+}^* \times{\mathbb{T}_d} \times \mathbb{R}^d)$, continuous on $\mathbb{R}_{+} \times{\mathbb{T}_d} \times \mathbb{R}^d$ such that $\mathfrak{g}$ is solution of the Vlasov equation, i.e. for all $(t,x,v) \in \mathbb{R}_+^*\times{\mathbb{T}_d}\times \mathbb{R}^d$
\begin{equation}
\label{Vlasov_equation}
 \left\{ \begin{array}{llll} \partial_t \mathfrak{g} (t,x,v) + v\cdot\nabla_x \mathfrak{g}(t,x,v) - \nabla_x \mathfrak{u}(t,x) \cdot\nabla_v f^{eq}(v) = \mathfrak{S}(t,x,v), \\
									\mathfrak{g}(0,x,v)=\mathfrak{g}_0(x,v).
\end{array}  \right. 
\end{equation}
If $\lambda> \lambda_0$ then for all $k\in \widehat{{\mathbb{T}_d}}$, there exists $C>0$,
\begin{equation}
\label{control_got_g}
\forall v \in \mathbb{R}^d,  \ \sup_{t\in \mathbb{R}_+} |e^{-\lambda t}\widehat{\mathfrak{g}}(t,k,v)| \leq C \mathfrak{d}(v).
\end{equation}
Furthermore, for all $ z \in \mathbb{C}$ with $\Im(z)>\lambda_0$ we have
\begin{multline}
\label{eq_post_Vlin}
\Lap \left[ \int_{\mathbb{R}^d} \widehat{\mathfrak{g}}(t,k,v) \textrm{d}v \right] (z)  =  -i \int_{\mathbb{R}^d} \frac{ \widehat{\mathfrak{g}_0}(k,v)}{v\cdot k-z} \textrm{d}v 
+  \Lap [\widehat{\mathfrak{u}}(t,k)](z) \int_{\mathbb{R}^d} \frac{ k\cdot \nabla_v f^{eq}(v)}{v\cdot k-z} \textrm{d}v \\- i \int_{\mathbb{R}^d}  \frac{ \Lap \left[ \widehat{\mathfrak{S} }(t,k,v) \right](z)}{v\cdot k-z} \textrm{d}v.
\end{multline}  
\end{proposition}
\begin{proof}

First, applying a space Fourier transform to \eqref{Vlasov_equation}, we get for all $(t,k,v) \in \mathbb{R}_+^* \times \widehat{{\mathbb{T}_d}} \times  \mathbb{R}^d$
\begin{equation}
\label{Vlasov_lin_Four}
\partial_t \widehat{\mathfrak{g} }(t,k,v) + iv\cdot k \  \widehat{\mathfrak{g} }(t,k,v) - i \widehat{\mathfrak{u} }(t,k) k\cdot \nabla_v f^{eq}(v) = \widehat{\mathfrak{S} }(t,k,v) . 
\end{equation}
Consequently, applying Duhamel formula, we get for all $(t,k,v) \in \mathbb{R}_+\times \widehat{{\mathbb{T}_d}} \times \mathbb{R}^d$
\begin{multline}
\label{g_Duhamel_fixed_mod}
 \widehat{\mathfrak{g}}(t,k,v) = e^{-ik\cdot vt} \widehat{\mathfrak{g}_0}(k,v) + i \int_{0}^t e^{-i k\cdot v(t-s)} \widehat{\mathfrak{u}}(s,k) k\cdot\nabla f^{eq}(v) \textrm{d}s \\
 + \int_{0}^t e^{-i k\cdot v(t-s)}  \widehat{\mathfrak{S} }(s,k,v) \textrm{d}s.
\end{multline}
So, we deduce, there exist $M,C>0$ such that, if $\lambda>\lambda_0$ then
\begin{align*}
| \widehat{\mathfrak{g}}(t,k,v)| &\leq  |\widehat{\mathfrak{g}_0}(k,v)| + \int_{0}^t e^{\lambda_0 s} M |k| |\nabla_v f^{eq}(v)| \textrm{d}s + \int_{0}^t e^{\lambda_0 s}  e^{-\lambda_0 s} |\widehat{\mathfrak{S} }(s,k,v) | \textrm{d}s, \\
 &\leq \mathfrak{d}(v) + t e^{\lambda_0 t}  (1+|k| M) \mathfrak{d}(v),\\
 &\leq C e^{\lambda t} \mathfrak{d}(v).
\end{align*}

We deduce of this last estimation, that for any fixed $v\in \mathbb{R}^d$ and for any $\lambda>\lambda_0$, $e^{-\lambda t}\widehat{g}(t,k,v)$ is continuous and bounded on $\mathbb{R}_+$. Consequently, we can apply a Laplace transform on \eqref{Vlasov_lin_Four} and get for all $z \in \mathbb{C}$ such that $\Im z > \lambda_0$ and $v\in \mathbb{R}^d$,
\begin{multline*} -i  z \Lap [\widehat{\mathfrak{g}}(t,k,v)](z) -  \widehat{\mathfrak{g}_0}(k,v) + iv\cdot k \Lap [\widehat{\mathfrak{g}}(t,k,v)](z) - i \Lap [\widehat{\mathfrak{u}}(t,k)](z) k\cdot \nabla_v f^{eq}(v)\\
 = \Lap \left[ \widehat{\mathfrak{S} }(t,k,v) \right](z). \end{multline*}
Since $\Im z>0$, this relation can be divided by $i(v\cdot k-z)$ to get for all $v\in \mathbb{R}^d$
\[ \Lap [\widehat{\mathfrak{g}}(t,k,v)](z)  =  -i \frac{ \widehat{\mathfrak{g}_0}(k,v)}{v\cdot k-z} +  \Lap [\widehat{\mathfrak{u}}(t,k)](z)  \frac{ k\cdot \nabla_v f^{eq}(v)}{v\cdot k-z} - i \frac{ \Lap \left[ \widehat{\mathfrak{S} }(t,k,v) \right](z)}{v\cdot k-z}.\]
Finally we conclude this proof integrating with respect to $v$ and applying Fubini Theorem (with the control \eqref{control_got_g}) to get for all $z\in \mathbb{C}$ with $\Im z>\lambda_0$
\[ \Lap \left[ \int_{\mathbb{R}^d} \widehat{\mathfrak{g}}
(t,k,v) 
\textrm{d}v \right] (z) =\int_{\mathbb{R}^d} \Lap \left[ \widehat{\mathfrak{g}}(t,k,v)  \right] (z) \textrm{d}v. \] 
\end{proof}

If we want to get a closed equation on $\mathfrak{u}$, we have to use Poisson equation
\begin{equation}
\label{Poisson}
\Delta_x \mathfrak{u}(t,x) = - \int_{\mathbb{R}^d} \mathfrak{g}(t,x,v) \textrm{d}v.
\end{equation}
Formally, applying a space Fourier transform and a Laplace transform we would get
\[ |k|^2 \Lap \left[\widehat{\mathfrak{u}}(t,k)\right] =  \Lap \left[\int_{\mathbb{R}^d} \widehat{ \mathfrak{g}}(t,k,v) \textrm{d}v\right].\] 
Consequently, applying \eqref{eq_post_Vlin}, we should get the following dispersion relation
\begin{equation}
\label{rel_disp_lin_gen}
\Lap \left[ \widehat{\mathfrak{u}}(t,k)\right] (z) D_k(z) = -\frac{i}{|k|^2}  \int_{\mathbb{R}^d} \frac{\widehat{\mathfrak{g}_0}(k,v)}{v\cdot k-z}  \textrm{d}v -\frac{i}{|k|^2}  \int_{\mathbb{R}^d} \frac{\Lap \left[ \widehat{\mathfrak{S}}(t,k,v) \right](z)}{v\cdot k-z}  \textrm{d}v,
\end{equation}
where $D_k$ is defined by \eqref{intro_defNkDk}.

\begin{proposition}
\label{exist_lin_gen} 
 Assume $\mathfrak{g}_0 \in C^1({\mathbb{T}_d}\times \mathbb{R}^d)$, $f^{eq}\in C^2(\mathbb{R}^d)$ and $\mathfrak{S}(t,x,v)\in C^1(\mathbb{R}_+^*\times{\mathbb{T}_d}\times \mathbb{R}^d)$  and there exist $\lambda_0>0$, $\mathfrak{d} \in C^0(\mathbb{R}^d)\cap L^1(\mathbb{R}^d)$ satisfying
\[ \forall (t,k,v) \in \mathbb{R}_+^*\times \widehat{{\mathbb{T}_d}} \times \mathbb{R}^d, \ e^{-\lambda_0 t} |\widehat{\mathfrak{S}}(t,k,v)|+|\widehat{\mathfrak{g}_0}(k,v)| + |\nabla_v f^{eq}(v)| \leq \mathfrak{d}(v).\]
Assume there exists $ \mathfrak{u} \in C^1(\mathbb{R}_+^*\times{\mathbb{T}_d})$ such that $e^{ -\lambda_0 t}\mathfrak{u}(t)$ is bounded on $\mathbb{R}_+^*\times{\mathbb{T}_d}$. Furthermore, assume that, for all $k \in \widehat{{\mathbb{T}_d}}\setminus \{0\}$, $\Lap \left[ \widehat{\mathfrak{u}}(t,k) \right](z)$ is a solution of \eqref{rel_disp_lin_gen} for $\Im z > \lambda_0$.
Assume there exists a finite set $\mathfrak{K}\subset \widehat{{\mathbb{T}_d}}$ such that
\[ \forall t\in \mathbb{R}, \forall v\in \mathbb{R}^d, \ k\in \widehat{{\mathbb{T}_d}}\setminus \mathfrak{K} \ \Rightarrow \ \widehat{\mathfrak{g}_0}(k,v)= \widehat{\mathfrak{S}}(t,k,v)=\widehat{\mathfrak{u}}(t,k)=0. \]
If we define $\mathfrak{g}$ by
\begin{multline}
\label{g_Duhamel}
\mathfrak{g}(t,x,v) = \sum_{k \in \mathfrak{K}} e^{ik\cdot(x-vt)} \widehat{\mathfrak{g}_0}(k,v) + i \int_{0}^t e^{i k\cdot(x-v(t-s))} \widehat{\mathfrak{u}}(s,k) k\cdot\nabla f^{eq}(v) \textrm{d}s \\+ \int_{0}^t e^{i k\cdot(x-v(t-s))}  \widehat{\mathfrak{S} }(s,k,v) \textrm{d}s ,
\end{multline}
then $\mathfrak{g}  \in  C^1(\mathbb{R}_{+}^* \times{\mathbb{T}_d} \times \mathbb{R}^d)$ is continuous on $\mathbb{R}_{+} \times{\mathbb{T}_d} \times \mathbb{R}^d$ and
$(\mathfrak{g},\mathfrak{u})$ is a solution of \eqref{VP_lin_gen}.
\end{proposition}
\begin{proof}

By construction of $\mathfrak{g}$ through Duhamel formula \eqref{g_Duhamel}, $\mathfrak{g}$ is obviously a continuous function on $\mathbb{R}_{+} \times{\mathbb{T}_d} \times \mathbb{R}^d$ and $C^1$ on $\mathbb{R}_{+}^* \times{\mathbb{T}_d} \times \mathbb{R}^d$. Furthermore, we may verify by a straightforward calculation that $\mathfrak{g}$ is solution of the Vlasov equation \eqref{Vlasov_equation}. Consequently, we just have to prove that $\mathfrak{g},\mathfrak{u}$ is solution of Poisson equation \eqref{Poisson}.

\medskip

However $\mathfrak{u}$ and $\mathfrak{g}$ satisfy assumptions of Proposition \ref{prop_anal_lin}, so we can apply it. Consequently, we know that if $\lambda>\lambda_0$ then $e^{-\lambda t}\int \widehat{\mathfrak{g}}(t,k,v) \textrm{d}v$ is continuous and bounded and that its Laplace transform satisfies \eqref{eq_post_Vlin}. But since $\Lap [\widehat{\mathfrak{u}}(t,k)]$ is a solution of the dispersion relation \eqref{rel_disp_lin}, we deduce that for all $z\in \mathbb{C}$ such that $\Im z> \lambda_0$ we have
\[ |k|^2 \Lap \left[ \widehat{\mathfrak{u}}(t,k)\right](z) =  \Lap \left[ \int_{\mathbb{R}^d} \widehat{g}(t,k,v) \textrm{d}v \right](z)  .\]
But it is well known that Laplace transform is injective on continuous functions with an exponential order (i.e. bounded by an exponential function), see Theorem $1.7.3$ in \cite{MR2798103}. Consequently, we have for all $t>0$
\[ |k|^2 \widehat{\mathfrak{u}}(t,k) = \int_{\mathbb{R}^d} \widehat{\mathfrak{g}}(t,k,v) \textrm{d}v .\]
Since space Fourier transform is also injective on regular functions, we have proven that $\mathfrak{u},\mathfrak{g}$ is a solution of Poisson equation \eqref{Poisson}.
\end{proof}

\medskip


%

\subsection{Proof of Propositions \ref{propo_existence_lin} and \ref{propo_existence_o2}} We now apply Proposition \ref{exist_lin_gen} for the proof of  Propositions \ref{propo_existence_lin} and \ref{propo_existence_o2}.
\begin{proof}[Proof of Proposition \ref{propo_existence_lin}] First, observe that if $k\in \widehat{{\mathbb{T}_d}}\setminus \left( \{ K \}\cup \{0\} \right)$ then for any $t>0$, we have $\widehat{\psi}(t,k)=0$. Indeed, since $\Lap \left[ \widehat{\psi}(t,k)\right] (z) $ is a solution of \eqref{rel_disp_lin}, we have
\[ D_k(z) \Lap \left[ \widehat{\psi}(t,k)\right] (z)  =0.  \] 
But, we have proven in Lemma \ref{Dk_cool} that $D_k(z)\neq 0$ if $\Im z$ is large enough. Consequently, $\Lap \left[ \widehat{\psi}(t,k)\right] (z)=0$ if $\Im z$ is large enough.
So we deduce by a classical criterion about Laplace transform (see Theorem $1.7.3$ in \cite{MR2798103}) that $\widehat{\psi}(t,k)=0$.

\medskip

We observe on \eqref{g_Duhamel_1} that $g$ is clearly a $C^{\infty}$ function on $\mathbb{R}_+^*\times{\mathbb{T}_d}\times \mathbb{R}^d$. Finally, we just need to apply Proposition \ref{exist_lin_gen} to prove that $(g,\psi)$ is a solution of \eqref{VPL}.
\end{proof}

\begin{proof}[Proof of Proposition \ref{propo_existence_o2}]
Let $\mathfrak{S}$ be defined by
\[ \mathfrak{S}(t,x,v) = \nabla_x \psi(t,x) \cdot \nabla_v g(t,x,v).\]
By construction, it is a $C^{\infty}$ function on $\mathbb{R}_+^*\times{\mathbb{T}_d}\times \mathbb{R}^d$. Since, space Fourier transform of $\psi$ is supported by $K$ (see proof of Proposition \ref{propo_existence_lin}), its space Fourier transform is supported by $K+K$. Furthermore, since $\lambda_1>2\lambda_0$, we can construct, by a straightforward estimation, a continuous function $\mathfrak{d}\in C^0(\mathbb{R}^d)\cap L^1(\mathbb{R}^d)$ such that
\[ \forall v\in \mathbb{R}^d, \ e^{-\lambda_1 t}|\widehat{\mathfrak{S}}(t,k,v)| \leq  \mathfrak{d}(v).\]
In particular, this estimation proves that the right member of \eqref{rel_disp_o2} is well defined if $\Im z\geq \lambda_1$.

\medskip

Now, as in Proposition \ref{propo_existence_lin}, we can first prove that space Fourier transform of $\mu$ is supported by $(K+K) \cup \{0\}$, then observe that $h$ is a $C^{\infty}$ function and conclude that $(h,\mu)$ is a solution of \eqref{VPL} by Proposition \ref{exist_lin_gen}.
\end{proof}

\section{Resolution and expansion of the linearized equation}
\label{section_linear_resolution}
\subsection{Introduction and statement of the result}
In Proposition \ref{propo_existence_lin}, we have proved that it is enough to solve dispersion relation \eqref{rel_disp_lin} to get a solution $(g,\psi)$ to linearized Vlasov-Poisson equation \eqref{VPL}. So the aim of this section is to solve this dispersion relation introducing most of the theoretical tools useful in the resolution of the second order relation \eqref{rel_disp_o2}. In particular, many of them deal with analytic function defined on {\it open sectors}, denoted $\Sigma_\alpha$, with $\alpha\in (0,\pi)$, and defined by  
\[  \Sigma_\alpha = \{  re^{i \beta } \ | \ -\alpha<\beta<\alpha \textrm{ and } r>0 \}. \]

The result we are going to establish in this section is the following.
\begin{proposition}
\label{resol_rel_disp_lin}
Assume $f^{eq}\in \mathscr{E}(\mathbb{R}^d)$ and $g_0$ satisfies Assumption \ref{assump}.
For all $\lambda \in \mathbb{R}$, for all $k\in K$, for all zero point $\omega$ of $D_k$ there exists a polynomial, denoted $P_{k,\omega}$, whose degree is strictly smaller than the multiplicity of $\omega$, $\alpha\in(0,\frac{\pi}2)$ and there exists $r_{k,\lambda}$ an analytic and bounded function on $\Sigma_{\alpha}$ such that the following expansion defines a solution of the dispersion relation \eqref{rel_disp_lin}
\[ \forall t \in \mathbb{R}_+^*, \forall x\in{\mathbb{T}_d}, \ \psi(t,x) = \sum_{k\in K} \sum_{\substack{ D_k(\omega)=0 \\
													\Im \omega \geq \lambda	}} P_{k,\omega}(t) e^{i(k\cdot x-\omega t)} + e^{ik\cdot x} e^{\lambda t} r_{k,\lambda}(t).\]
\end{proposition}

\medskip

This proposition will be proven at the end of this section. First, we introduce some notations and many useful theoretical tools.

\subsection{Definition of $N_k$ and theoretical tools} The right member of the dispersion relation \eqref{rel_disp_lin} is very important in our study. We denote it $N_k(z)$. More precisely, it is an analytic function defined, when $\Im z>0$ by
\[ N_k(z) = -\frac{i}{|k|^2} \int \frac{\widehat{g_0}(k,v)}{v\cdot k-z} \textrm{d}v.\]
In the first part of this proof we study the regularity and the behavior of $D_k$ and $N_k$. However, we need to introduce some classical results on Laplace transform.

\medskip

First, consider the following Theorem that is very useful to invert Laplace transforms and to control it.

\begin{theorem} (Analytic representation)\\
\label{theo_anal_rep} 
Let $\alpha \in(0,\frac{\pi}2)$, $\lambda_0\in \mathbb{R}$ and $q:i(\lambda_0,\infty)\to \mathbb{C}$. The following assertions are equivalent:
\begin{enumerate}[(i)]
\item There exists a holomorphic function $f:\Sigma_{\alpha} \to \mathbb{C}$ such that
\[  \forall 0<\beta<\alpha, \sup_{z\in \Sigma_\beta} |e^{-\lambda_0 z} f(z)|<\infty \textrm{ and } \forall \lambda>\lambda_0, \ q(i\lambda)=\Lap [f](i\lambda).\]
\item The function $q$ has a holomorphic extension $\widetilde{q}:i\lambda_0 + i\Sigma_{\alpha+\frac{\pi}2}\to \mathbb{C}$ such that
\[ \forall 0<\gamma<\alpha, \ \sup_{z \in i\lambda_0 + i\Sigma_{\gamma+\frac{\pi}2}} |(z-i\lambda_0)\widetilde{q}(z)|<\infty. \]
\end{enumerate}
\end{theorem}
\begin{proof}
See Theorem $2.6.1$ in \cite{MR2798103} page $87$.
\end{proof}

\begin{remark}
Note that if $e^{-\lambda_0t}f(t)$ is bounded on $\mathbb{R}_+^*$ then for $\lambda>\lambda_0$, $e^{-\lambda t}f(t)\in L^1(\mathbb{R}_+)$ and so $\mathcal{L}[f]$ is well defined for $\Im(z)>\lambda_0$, which is the set $i\lambda_0+i\Sigma_{\frac{\pi}2}$.
\end{remark}
There is a direct corollary of the proof of Theorem \ref{theo_anal_rep} that is useful in our study.
\begin{corollary} Assume that conclusion of Theorem  \ref{theo_anal_rep} holds. Then for all $0<\gamma<\beta<\alpha$, we have
\[  \sup_{z \in i\lambda_0 + i\Sigma_{\gamma+\frac{\pi}2}} |(z-i\lambda_0)\widetilde{q}(z)| \leq \frac1{\sin(\beta-\gamma)} \sup_{z\in \Sigma_\beta} |e^{-\lambda_0 z} f(z)|. \]
\end{corollary}

\medskip

Then, we observe that $D_k$ and $N_k$ are defined through a integral operator whose kernel is $\frac1{v\cdot k-z}$. The following lemma links this operator to more classical ones.
\begin{lemma}
\label{Hilb_to_Four} Let $f\in L^1(\mathbb{R}^d)$ and $k\in \mathbb{R}^d\setminus \{0\}$ then for all $z\in \mathbb{C}$ with $\Im z>0$
\[ \int_{\mathbb{R}^d} \frac{f(v)}{k\cdot v-z} \textrm{d}v = i \Lap\left[ \Four [f](kt) \right](z). \]
\end{lemma}
\begin{proof}
First, remark that since $f\in L^1(\mathbb{R}^d)$, $t\rightarrow \Four [f](k t)= \int_{\mathbb{R}^d} f(v) e^{-itv\cdot k} \textrm{d}v$ is a continuous and bounded function, so its Laplace transform is well defined if $\Im z>0$. Now, consider the following function
\[ F(t) =  \int_{\mathbb{R}^d} \frac{f(v)}{k\cdot v-z} e^{-i (k\cdot v -z)t}\textrm{d}v.\]
Since $f\in L^1(\mathbb{R}^d)$, it is a regular function and we have
\[ F'(t) = -i \int_{\mathbb{R}^d} f(v) e^{-i (k\cdot v -z)t}\textrm{d}v = -i \Four [f](kt) e^{izt}. \]
But, since $\Im z > 0$ we observe that $F(t)$ goes to $0$ when $t$ goes to $+\infty$. Consequently, we get
\[ F(0) = -\int_0^\infty F'(t) \textrm{d}t = i \int_0^{\infty} \Four [f](kt) e^{izt} \textrm{d}t =  i \Lap\left[ \Four [f](kt) \right](z). \]
\end{proof}

\medskip

%
%

\subsection{Estimations for $D_k$ and $N_k$}
With Lemma \ref{Hilb_to_Four}, we can write $D_k$ and $N_k$ as Laplace transforms. So, in the following proposition, we can prove their analyticity using the analytic representation theorem (Theorem \ref{theo_anal_rep}). In particular, we prove and extend Remark \ref{remark_NkDk}.
\begin{proposition}
\label{Dk_cool}
If $f^{eq}\in \mathscr{E}(\mathbb{R}^d)$ then
\begin{itemize}
\item for all $k\in  \mathbb{R}^d \setminus \{0\}$, $D_k$ is an entire function,
\item there exists $\alpha \in (0,\frac{\pi}2)$ such that for all $0<\gamma<\alpha$ and for all $\lambda_0\in \mathbb{R}$ there exists $C>0$ satisfying
\[ \forall k\in  \mathbb{R}^d \setminus \{0\},\forall z\in i|k|\lambda_0+i\Sigma_{\gamma +\frac{\pi}2}, \ |D_k(z)-1|\leq \frac{C}{|k| |z-i|k|\lambda_0|}. \]
\end{itemize}
\end{proposition}
\begin{proof} Since $f\in \mathscr{S}(\mathbb{R}^d)$, we have $k\cdot\nabla_v f^{eq} \in L^1(\mathbb{R}^d)$. Consequently, $D_k$ is well defined as an analytic function on $\{ z \in \mathbb{C} \ | \ \Im z>0 \}$. Furthermore, we can apply Lemma \ref{Hilb_to_Four} to get for $\Im z>0$
\[  D_k(z) =1- \frac{i}{|k|^2} \Lap  \left[  \Four\left[k\cdot\nabla_v f^{eq}\right](kt)  \right] (z).\]
Then we define $e_k = \frac{k}{|k|}$  and we get by the change of variable $t'=|k|t$
\begin{equation*}
\begin{split}
 &\Lap  \left[  \Four\left[k\cdot \nabla_v f^{eq}\right](kt)  \right] (z) \\
 =& \int_0^\infty \Four\left[|k|e_k\cdot\nabla_v f^{eq}\right](|k|e_kt)e^{i\frac{z}{|k|}|k|t}\textrm{d}t \\
 =&\int_0^\infty \Four\left[e_k\cdot\nabla_v f^{eq}\right](e_kt')e^{i\frac{z}{|k|}t'}\textrm{d}t',
 \end{split}
\end{equation*}
so that
\begin{equation}
\label{Dk_after_variable_change}
 D_k(z) =1- \frac{i}{|k|^2} \Lap  \left[  \Four \left[e_k \cdot\nabla_v f^{eq}\right]( e_k t) \right] \left(\frac{z}{|k|}\right).
\end{equation}
Now, using Theorem \ref{theo_anal_rep}, we are going to prove this Laplace transform defines an entire function and we are going  to control it. 

\medskip

Since $f^{eq}\in \mathscr{E}(\mathbb{R}^d)$, it extends to an analytic function and there exists $\alpha \in (0,\frac{\pi}2)$ such that for all $\beta \in (0,\alpha)$ and for all $\lambda\in \mathbb{R}$, there exist $C>0$ such that
\[  \forall z\in \Sigma_{\beta} \mathbb{R}^d, \ |\Four [f^{eq}](z)|\leq C |e^{-\lambda z}|.\]
Consequently, we get
\[ \forall z\in \Sigma_{\beta}, \ |\Four \left[e_k\cdot\nabla_v f^{eq}\right](e_k z)| =  |iz \Four [f^{eq}](z)| \leq C  |z| e^{-\lambda \Re z}. \]

Finally, we have proven that for all $\lambda_0 \in \mathbb{R}$, there exists a constant $M>0$ (independent of $e_k$) such that
\[ \forall z\in \Sigma_{\beta}, \  |\Four \left[e_k\cdot\nabla_v f^{eq}\right](e_k z)| \leq M |e^{-\lambda_0 z}|.\]

\medskip

Applying Theorem \ref{theo_anal_rep} and its corollary, we have proven that $$\Lap [\Four \left[e_k\cdot\nabla_v f^{eq}\right](e_k t)]$$ is an entire function and that for all $\gamma \in (0,\alpha)$ and all $\lambda_0\in \mathbb{R}$, there exists $M>0$ (associated to $\beta = \frac{\alpha+\gamma}2$) such that
\[  \forall z \in i\lambda_0 + i \Sigma_{\gamma+\frac{\pi}2} , \ \left|\Lap \left[ \Four \left[e_k\cdot\nabla_v f^{eq}\right](e_k t)\right] (z)\right| \leq \frac{M}{\sin(\frac{\alpha-\gamma}2)} \frac1{|z-i\lambda_0|}. \]
Finally, we deduce directly the result from formula \eqref{Dk_after_variable_change}:
\begin{multline*}
 \left|D_k(z) -1\right|= \left|\frac{i}{|k|^2} \Lap  \left[ \Four \left[e_k \cdot \nabla_v f^{eq}\right]( e_k t) \right] \left(\frac{z}{|k|}\right)\right| \\
 \leq  \frac{M}{|k|^2\sin(\frac{\alpha-\gamma}2)} \frac1{|\frac{z}{|k|}-i\lambda_0|}  =  \frac{M}{|k|\sin(\frac{\alpha-\gamma}2)} \frac1{|z-i|k|\lambda_0|}.
 \end{multline*}

\end{proof}

\begin{corollary} \label{zero_loc_Dk} If $f^{eq}\in \mathscr{E}(\mathbb{R}^d)$ then there exists $\alpha\in(0,\frac{\pi}2)$ such that for all $\lambda_0\in \mathbb{R}$ and $\gamma\in (0,\alpha)$, there exists $c>0$ such that for all $k\in \mathbb{R}^d \setminus\{ 0\}$ we have
\[  \{ z \in \mathbb{C} \ | \ D_k(z) = 0 \} \subset  i|k|\lambda_0 - \left(\mathbb{D}(0,\frac{c}{|k|}) \cup \overline{i\Sigma_{\frac{\pi}2-\gamma}}\right) .  \]
\end{corollary}
\begin{proof}
Indeed, we have either $z\in  i|k|\lambda_0+i\Sigma_{\gamma +\frac{\pi}2}$, so that $1= |D_k(z)-1|\leq \frac{C}{|k| |z-i|k|\lambda_0|}$ and thus $z\in i|k|\lambda_0 - \mathbb{D}(0,\frac{C}{|k|})$. Otherwise, 
$z\in \mathbb{C}\setminus\left\{ i|k|\lambda_0+i\Sigma_{\gamma +\frac{\pi}2}\right\}$, that is $z= i|k|\lambda_0+ire^{i\delta}$,\ 
$\delta \in[\frac{\pi}{2}+\gamma,2\pi-\frac{\pi}{2}-\gamma]$, and thus $\pi-\delta\in [-\frac{\pi}{2}+\gamma,-\gamma+\frac{\pi}{2}]$, which leads to $z= i|k|\lambda_0-ire^{-i\tilde{\delta}}$, with
$\tilde{\delta}=\pi-\delta\in [-\frac{\pi}{2}+\gamma,-\gamma+\frac{\pi}{2}]$.
\end{proof}

\begin{corollary} \label{cor_Dk_cool}
If $f^{eq} \in \mathscr{E}(\mathbb{R}^d)$ then for all $k\in  \widehat{{\mathbb{T}_d}}\setminus \{0\}$, $D_k$ is an entire function and for all $\lambda\in \mathbb{R}$, $\{\omega \in \mathbb{C} \ | \ D_k(\omega)=0 \textrm{ and } \Im \omega \geq \lambda \}$ is a finite set.
 \end{corollary}
\begin{proof}
In Proposition \ref{Dk_cool}, we have proved that $D_k$ is an entire function and it can be directly deduced from its Corollary \ref{zero_loc_Dk} that $\{ z \in \mathbb{C} \ | \ D_k(z) = 0 \textrm{ and } \Im \omega \geq \lambda \}$ is bounded. Consequently, since zero points of $D_k$ are isolated, it is a finite set.
\end{proof}

It is very natural to adapt this result to $N_k$. More precisely, we deduce the following proposition.
\begin{proposition}
\label{Nk_nice} For all $k \in \widehat{{\mathbb{T}_d}}\setminus\{0\}$, $N_k$ is an entire function and there exists $\alpha \in  (0,\frac{\pi}2)$ such that for all $\lambda_0\in \mathbb{R}$ and for all  $\beta \in (0,\alpha)$, we have
\[ \sup_{z\in i \lambda_0 + i\Sigma_{\beta+\frac{\pi}2}} |N_k(z)| |z-i\lambda_0| < \infty. \]
\end{proposition}

\begin{proof} Since $v\mapsto \widehat{g_0}(k,v)\in \mathscr{S}(\mathbb{R}^d)$, $N_k$ defines naturally an analytic function for $\Im z>0$. Furthermore, applying Lemma \ref{Hilb_to_Four} we know that for $\Im z>0$
\[  N_k(z) = \frac{1}{|k|^2} \Lap  \left[  \Four\left[ \widehat{g_0}(k,v) \right](kt)  \right] (z).\]
Since $v\mapsto \widehat{g_0}(k,v) \in \mathscr{E}(\mathbb{R}^d) $, $t\mapsto \Four\left[ \widehat{g_0}(k,v) \right](kt)$ is an entire function and there exists $\alpha\in (0,\frac{\pi}2)$ such that for all $\beta \in (0,\alpha)$ and all $\lambda_0 \in \mathbb{R}$, $t\mapsto |e^{\lambda_0 t}\Four\left[ \widehat{g_0}(k,v) \right](kt)|$ is bounded on $\Sigma_\beta$.
Consequently, we deduce of Theorem \ref{theo_anal_rep}, that for all $\lambda_0 \in \mathbb{R}$, $N_k$
 has a holomorphic extension on $i\lambda_0 + i\Sigma_{\alpha+\frac{\pi}2}$ such that
\[ \forall \beta\in(0,\alpha), \ \sup_{z \in i\lambda_0 + i\Sigma_{\gamma+\frac{\pi}2}} |(z-i\lambda_0)N_k(z)|<\infty. \]
Finally, since $N_k$ is analytic on $i\lambda_0 + i\Sigma_{\alpha+\frac{\pi}2}$ for all $\lambda_0\in \mathbb{R}$, it is an entire function.
\end{proof}

\subsection{A theoretical tool for the control of $\mathcal{L}^{-1}[N_k/D_k]$}
Now, we introduce a general criterion to invert Laplace transform and get an asymptotic expansion.
\begin{lemma} 
\label{lemma_key}
Let $R \in \mathcal{H}(\mathbb{C})$ be an entire function and $N$ be a meromorphic function defined on $\mathbb{C}$. If there exists $\alpha \in (0,\frac{\pi}2)$ such that
\[\exists C>0, \ \sup_{z \in i\Sigma_{\alpha + \frac{\pi}2} } |zR(z)| + |zN(z)| <C,  \]
then for any $\lambda\in \mathbb{R}$, there exist $\beta \in (0,\alpha)$ and a function $r  \in \mathcal{H}(\Sigma_{\beta})$ analytic and bounded on $\Sigma_{\beta}$ such that if $\Im z$ is large enough then
\[ \frac{N(z)}{1-R(z)} = \Lap \left[   \sum_{\substack{ \omega\in Z\\
													\Im \omega \geq \lambda	}} P_{\omega}(t) e^{-i \omega t} + e^{\lambda t} r(t) \right](z)      ,\]
where $Z$ is the set of poles of $\frac{N(z)}{1-R(z)}$,
$$
P_{\omega} = \sum_{k=0}^{n_\omega-1}\frac{a_{k+1,\omega} (-i)^{k+1}}{k!} X^{k}
$$
 is the polynomial whose coefficients are defined by the expansion of $\frac{N}{1-R}$ in $z=\omega$
\[ \frac{N(z)}{1-R(z)} \mathop{=}\limits_{z\to \omega}  \sum_{j=1}^{n_\omega} \frac{a_{j,\omega}}{(z-\omega)^j} + \mathcal{O}(1)  .\]
\end{lemma}
\begin{remark}
In the application, for the proof of Proposition \ref{resol_rel_disp_lin}, $N$ will be entire (thus meromorphic), but for the second order case, in the next section, we will really need that $N$ is meromorphic.
\end{remark}
\begin{proof}[Proof of Lemma \ref{lemma_key}]  Many geometrical objects are going to be introduced in this proof. The reader can refer to Figure \ref{fig:serrure} to an illustration of these constructions.\\
First observe that to prove the lemma, we can assume that $\lambda$ is negative enough. In particular we assume that $\lambda<-2C$.\\
By construction, if $|z|\ge2C$ and $z \in \overline{i\Sigma_{\alpha + \frac{\pi}2}}$ then $|1-R(z)|\ge1-|R(z)|>1-\frac{C}{|z|}\ge \frac12$. Consequently, all zero points of $(1-R)$  belong to $\mathbb{D}(0,2C)\cup -i\Sigma_{\frac{\pi}2-\alpha}$ (note that $-i\Sigma_{\frac{\pi}2-\alpha}{\setminus\{ 0\}}=\left({\overline{i\Sigma_{\alpha + \frac{\pi}2}}}\right)^c$). Since the poles of $N$ lie on $-i\Sigma_{\frac{\pi}2-\alpha}$, 
the poles of $\frac{N(z)}{1-R(z)}$ lie on $\mathbb{D}(0,2C)\cup -i\Sigma_{\frac{\pi}2-\alpha}$. In particular, the set of its poles with an imaginary part larger than or equal to $\lambda$ is finite (see Figure \ref{fig:serrure}).\\
Now consider the following rational fraction 
\[ Q(z) = \sum_{\substack{ \omega\in Z \\ 
													\Im \omega \geq \lambda	}}  \sum_{j=1}^{n_{\omega}} \frac{a_{j,\omega}}{(z-\omega)^j} .\]
We introduce $r_1>0$ such that we have
\[ \mathbb{D}(0,2C) \cup  \left(\{ z\in \mathbb{C} \ | \ \Im z \geq \lambda \}\cap -i\Sigma_{\frac{\pi}2-\alpha} \right)\subset \mathbb{D}(i\lambda,r_1).\]
	Now, we observe that there exists $\beta \in (0,\alpha)$ such that $\frac{N}{1-R}-Q$ is a continuous function on $\overline{i\lambda + i\Sigma_{\beta +\frac{\pi}2}}$. Indeed, it is a meromorphic function whose poles lie on $\{ z\in \mathbb{C} \ | \ \Im z < \lambda \} \cap -i\Sigma_{\frac{\pi}2-\alpha}$ and are isolated, and thus we can choose such $\beta$ (small enough). Consequently, there exists $M>0$ such that
	\[ \forall z \in  \mathbb{D}(i\lambda,r_1) \cap \left( i\lambda+i\Sigma_{\beta+\frac{\pi}2} \right), \ \left|  \frac{N(z)}{1-R(z)} - Q(z) \right|<M\leq \frac{Mr_1}{|z-i\lambda|}.\]
Furthermore since $zN(z)$ is bounded on $ i\Sigma_{\alpha + \frac{\pi}2} $, there exists $M_1>0$ such that
\[ \forall z \in i\lambda+i\Sigma_{\beta+\frac{\pi}2}, \ |N(z)| \leq \frac{M_1}{|z-i\lambda|}. \]
Indeed, we distinguish the case $z\in i\Sigma_{\alpha+\frac{\pi}2}\cap  \mathbb{D}^c(0,2|\lambda|)$, for which there exists $C>0$ such that
$$
|N(z)|\le \frac{C}{|z|}\le  \frac{C}{|z-i\lambda|}\frac{|z-i\lambda|}{|z|}\le  \frac{C}{|z-i\lambda|}\frac{3|\lambda|}{2|\lambda|}\le  \frac{M_1}{|z-i\lambda|},
$$
and the case $z\in \left(i\lambda+i\Sigma_{\beta+\frac{\pi}2}\right) \cap \left( ( i\Sigma_{\alpha+\frac{\pi}2})^c\cup \mathbb{D}(0,2|\lambda|) \right)$ which is a bounded set ensuring
$$
|z-i\lambda||N(z)|\le M_1.
$$
Consequently, by construction of $r_1$, we get $|1-R(z)|\ge \frac{1}2$, when $|z|\ge 2C$ and $z\in i\Sigma_{\alpha + \frac{\pi}2}$ [so, in particular when $z\in i\Sigma_{\alpha + \frac{\pi}2}\cap  \mathbb{D}^c(i\lambda,r_1) \cap  \left(i\lambda+i\Sigma_{\beta+\frac{\pi}2} \right)$]
and $|1-R(z)|\ge c$, with $c>0$, when $z\in \left(i\Sigma_{\alpha + \frac{\pi}2}\right)^c \cap  \left(i\lambda+i\Sigma_{\beta+\frac{\pi}2} \right)$ (which is a bounded set)
 [so, in particular  when $z\in\left( i\Sigma_{\alpha + \frac{\pi}2}\right)^c\cap  \mathbb{D}^c(i\lambda,r_1) \cap  \left(i\lambda+i\Sigma_{\beta+\frac{\pi}2} \right)$]
 and thus
\[ \forall z \in \left(i\lambda+i\Sigma_{\beta+\frac{\pi}2} \right) \cap  \mathbb{D}^c(i\lambda,r_1), \  \left| \frac{N(z)}{1-R(z)} \right|  \leq \frac{\max(2,1/c)M_1}{|z-i\lambda|} .\]
Finally, since $Q$ is a rational fraction whose poles lie on $\mathbb{D}(i\lambda,r_1)$ and vanishing as $z$ goes to $\infty$, the function $z\rightarrow (z-i\lambda)Q(z)$ is bounded on  $\mathbb{D}^c(i\lambda,r_1)$
and thus there exists $M_2>0$ such that
\[ \forall z \in \left(i\lambda+i\Sigma_{\beta+\frac{\pi}2} \right) \cap  \mathbb{D}^c(i\lambda,r_1), \  \left| Q(z) \right|  \leq \frac{M_2}{|z-i\lambda|} .\]
Then we get a constant $M_3>0$ such that
	\[ \forall z \in    i\lambda+i\Sigma_{\beta+\frac{\pi}2}, \ \left|  \frac{N(z)}{1-R(z)} - Q(z) \right|<\frac{M_3}{|z-i\lambda|}.\]

Applying Theorem \ref{theo_anal_rep} to $\frac{N}{1-R}-Q$, we get an analytic and bounded function $t\rightarrow e^{-\lambda t}e^{\lambda t}r(t) = r(t)$ on $\Sigma_{\gamma}$ (with $\gamma=\frac{\beta}2)$, 
such that
 \[  \Lap \left[ e^{\lambda t}  r(t) \right](z)= \frac{N(z)}{1-R(z)}-Q(z).\]
To conclude the proof of Lemma \ref{lemma_key}, 
we just need to determine the invert Laplace transform of $Q$. But we get, by straightforward calculation,
 \[ Q(z) = \Lap \left[  \sum_{\substack{  \omega\in Z\\ 
													\Im \omega \geq \lambda	}}   P_{\omega}(t) e^{-i \omega t} \right](z).\]
\end{proof}

\begin{figure}
\includegraphics[width=0.8\linewidth]{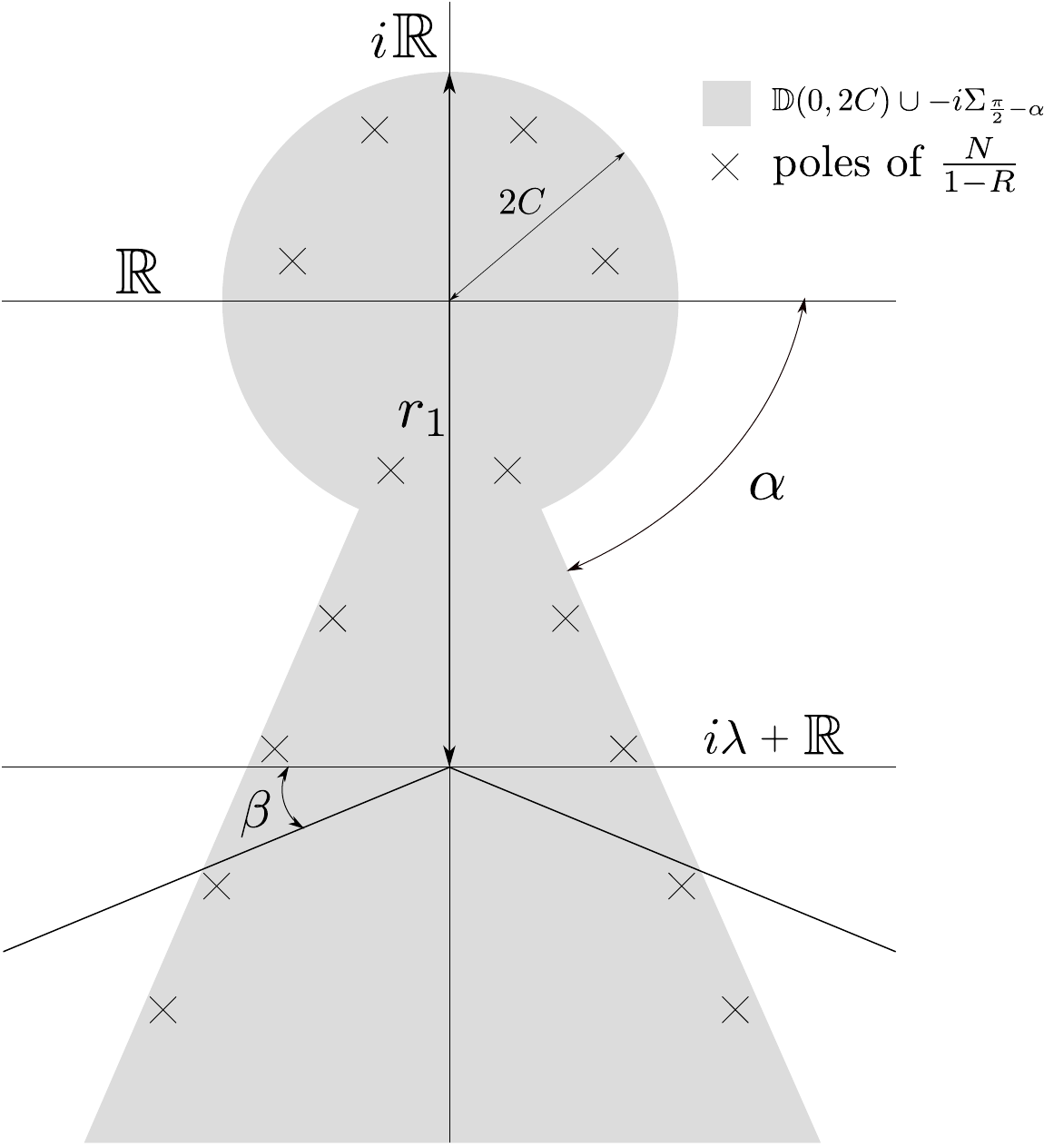}
\caption{An illustration of the geometrical constructions introduced in the proof of Lemma \ref{lemma_key}.}
\label{fig:serrure}
\end{figure}
\subsection{Proof of Proposition \ref{resol_rel_disp_lin}}
Finally we can prove the result stated at the beginning of this section.

\begin{proof}[Proof of Proposition \ref{resol_rel_disp_lin}] In Proposition \ref{Dk_cool} and \ref{Nk_nice} we have proven that we can apply Lemma \ref{lemma_key} with $D_k=1-R$ and $N=N_k$, taking $\lambda_0=0$. But the result of this lemma is exactly the expansion of  Proposition \ref{resol_rel_disp_lin}.
\end{proof}

\begin{remark}
As we use only $\lambda_0=0$ in Proposition \ref{Dk_cool} and \ref{Nk_nice} for the proof of Proposition \ref{resol_rel_disp_lin}, we may wonder of we could use a weaker assumption on $f^{eq}$ for getting the estimate on $D_k$
for example. Indeed, that estimate derives from Theorem \ref{theo_anal_rep} for $\lambda_0=0$ and so the weaker assumption could be the hypothesis of (i) in Theorem \ref{theo_anal_rep}  for $\lambda_0=0$. 
However, we also need to have that $D_k$ is entire, and there we have used Theorem \ref{theo_anal_rep} {\emph{for all}} $\lambda_0\in \mathbb{R}$.

\end{remark}

\section{Resolution and expansion of the second order equation}
\label{sec:4}
\subsection{Introduction and statement of the result} 
In Proposition \ref{propo_existence_o2}, we have proven that it is enough to solve dispersion relation \eqref{rel_disp_o2} to get a solution $(h,\mu)$ to second order linearized Vlasov-Poisson equation \eqref{VPL2}. 
So this section is devoted to the resolution of this second order dispersion relation, following the strategy established for the first order dispersion relation in the previous section,  
by proving the following proposition, which permits to complete the proof of our main result, Theorem \ref{The_th}.
\begin{proposition}
\label{resol_rel_disp_o2}  
Assume $f^{eq}\in \mathscr{E}(\mathbb{R}^d)$ and $g_0$ satisfies Assumption \ref{assump}. Consider the solution $(g,\psi)$ of \eqref{VPL} given by Proposition \ref{resol_rel_disp_lin} and Proposition \ref{propo_existence_lin}. Then there exists a solution $\mu$ of the dispersion relation \eqref{rel_disp_o2} whose expansion is detailed in Theorem \ref{The_th}.
\end{proposition}

\subsection{Definition of $\mathcal{N}^1_{k}$ and $\mathcal{N}^2_{k}$ (the right hand side)}
We will first look for the right hand side of the second order dispersion relation, that was $N_k$ in the first order case.

Let $k\in (K+K) \setminus \{0\}$. By looking at the second order dispersion relation \eqref{rel_disp_o2}, we can assume, without loss of generality, that there exist $k_1,k_2 \in K$ such that $k_1+k_2=k$ and
\[ \widehat{\nabla_x \psi \cdot \nabla_v g}(t,k,v) = i \widehat{\psi}(t,k_1) k_1\cdot\nabla_v \widehat{g}(t,k_2,v).\]
Consequently, we can determine more precisely the right member of \eqref{rel_disp_o2}. However, to be rigorous we need to prove that our integrals are convergent. Indeed, as in Proposition \ref{propo_existence_o2}, there exists $\lambda_0\in \mathbb{R}$ such that for any $\lambda>\lambda_0$, we can construct a continuous and integrable function $\mathfrak{d}\in C^0(\mathbb{R}^d)\cap L^1(\mathbb{R}^d)$ such that 
\[ \forall (t,v)\in \mathbb{R}_+^* \times \mathbb{R}^d, \ |\widehat{\psi}(t,k_1)|e^{-\lambda t} + e^{-\lambda t}  |k_1\cdot\nabla_v \widehat{g}(t,k_2,v)| \leq \mathfrak{d}(v) .\]

\noindent Consequently, if $\Im z>2\lambda_0$, we can apply Lemma \ref{Hilb_to_Four} to prove that

\begin{equation*}
\begin{split}
 &\int_{\mathbb{R}^d} \frac{\Lap \left[ i \widehat{\psi}(t,k_1) k_1\cdot\nabla_v \widehat{g}(t,k_2,v) \right](z) }{v\cdot k -z} \ \textrm{d}v\\ 
= &\ i\Lap  \left[ \Four \left[\Lap \left[  i \widehat{\psi}(t,k_1) k_1\cdot\nabla_v \widehat{g}(t,k_2,v) \right](z) \right](kt)\right](z)\\
= &\  i\int_0^\infty \int_{\mathbb{R}^d}\int_0^\infty  i \widehat{\psi}(t,k_1) k_1\cdot\nabla_v \widehat{g}(t,k_2,v) e^{izt}\textrm{d}t \ e^{-i(k \tau)\cdot v}\ \textrm{d}v \ e^{i z \tau}\ {\rm d \tau}\\
= &\ - i \int_0^\infty \int_0^\infty   \widehat{\psi}(t,k_1) (k_1\cdot k) \tau  \Four \widehat{g}(t,k_2,k\tau) \  e^{i z (\tau+t)} \textrm{d}t \ {\rm d \tau} \\
= &\ -i  (k_1\cdot k)   \int_0^{\infty}  \int_{0}^{s}      \widehat{\psi}(t,k_1)  (s-t)  \Four \widehat{g}(t,k_2,k(s-t))                 \textrm{d}t  \     e^{i z s}   \textrm{d}s\\
= &\ -i (k_1\cdot k) \Lap \left[ \int_{0}^t   \widehat{\psi}(\tau,k_1)   (t-\tau)  \Four \widehat{g}(\tau,k_2,k(t-\tau)) {\rm d \tau} \right](z),
\end{split}
\end{equation*}
where we have used the change of variable $\tau+t \leftarrow s$ and the notation $$
 \Four \widehat{g}(t,k,\xi) =  \Four [\widehat{g}(t,k,v)](\xi).$$
Furthermore,  using definition of $g$ (see \eqref{g_Duhamel_1}), we can precise $\Four \widehat{g}(\tau,k_2,k(t-\tau))$. Indeed, we start from
$$
\hat{g}(t,k,v) = e^{-ik\cdot vt} \widehat{g_0}(k,v) + i \int_{0}^t e^{-i k\cdot v(t-s)} \widehat{\psi}(s,k) k\cdot\nabla_v f^{eq}(v) \textrm{d}s,
$$
so that we have
\[ \Four \widehat{g}(t,k_2,\xi) = \Four [\widehat{g_0}(k_2,v)](\xi +  k_2t)  + i \int_{0}^t \widehat{\psi}(s,k_2) \Four \left[ k_2.\nabla_v f^{eq} \right](\xi + k_2 (t-s)) \textrm{d}s.\]
Consequently, we get
\begin{multline*} 
\Four \widehat{g}(\tau,k_2,k(t-\tau)) =   \Four [\widehat{g_0}(k_2,v)](k t +  (k_2-k)\tau)  \\
- i\int_{0}^{\tau} \widehat{\psi}(s,k_2)    \Four \left[ k_2\cdot\nabla_v f^{eq} \right](k (t -\tau) + k_2 (\tau-s)) \textrm{d}s.
\end{multline*}

\medskip

So we have two numerators to study for this dispersion relation. On the one hand, we have
\begin{equation}  \mathcal{N}^1_{k}(z) :=  - \frac{k_1\cdot k}{|k|^2} \Lap \left[ F^1_k(t) \right] (z),
\label{defNk1}
\end{equation}
with
\[ F^1_k(t) :=  \int_{0}^t   \widehat{\psi}(\tau,k_1)   (t-\tau) \Four [\widehat{g_0}(k_2,v)](k t +  (k_2-k)\tau) {\rm d\tau}.\]
On the other hand, we have
\begin{equation}\mathcal{N}^2_{k}(z) :=   i\frac{k_1\cdot k}{|k|^2} \Lap \left[ F^2_k(t) \right] (z),
\label{defNk2}
\end{equation}
with 
\[ F^2_k(t) = \int_{0}^t   \widehat{\psi}(\tau,k_1)   (t-\tau)  \int_{0}^{\tau} \widehat{\psi}(s,k_2)    \Four \left[ k_2\cdot \nabla_v f^{eq} \right](k (t -\tau) + k_2 (\tau-s)) \textrm{d}s \ {\rm d\tau}. \]
So with these notations,the  dispersion relation \eqref{resol_rel_disp_o2} may be written as, for all $z$ such that  $\Im z>2\lambda_0$,
\begin{equation}
\label{rel_disp_o2_nice}
\Lap \left[  \widehat{\mu}(t,k)	  \right](z) D_k(z) =   \mathcal{N}^1_{k}(z) +  \mathcal{N}^2_{k}(z) .
\end{equation}

\medskip

\subsection{Estimates for $\mathcal{N}^1_{k}$ and $\mathcal{N}^2_{k}$}
We are going to apply the same strategy as for the resolution of the first order dispersion relation. It will be solve using Lemma \ref{lemma_key}. The denominator has been studied in Proposition \ref{Dk_cool}. The following lemma describes the regularity and the behavior of the numerators $\mathcal{N}^1_{k}$ and $\mathcal{N}^2_{k}$.

\begin{lemma}
\label{lem_belle_N_is_nice}
The function $\mathcal{N}^1_{k}$, $\mathcal{N}^2_{k}$ have a meromorphic continuation and there exist $\widetilde{\lambda} \in \mathbb{R}$ and $\beta \in (0,\frac{\pi}2)$ such that
\[   \sup_{z \in i\widetilde{\lambda} +i\Sigma_{\beta+\frac{\pi}2}} |z-i\widetilde{\lambda}| \left[ |\mathcal{N}^1_{k}(z)| + |\mathcal{N}^2_{k}(z)| \right]<\infty. \]
\noindent $\bullet$ If there exists $\gamma\in (0,1)$ such $k_2 = -\gamma k_1$ then the poles of $\mathcal{N}^1_{k}$ are the points $\omega\in \mathbb{C}$ such that $\omega=\omega_1\frac{|k|}{|k_1|}$ where $D_{k_1}(\omega_1)=0$ and its multiplicity is smaller than or equal to $n_{k_1,\omega_1}+1$. $\mathcal{N}^2_{k}$ has two kinds of poles. On the one hand, there are the points $\omega\in \mathbb{C}$ such that $\omega=\omega_1+\omega_2$ where $D_{k_1}(\omega_1)=D_{k_2}(\omega_2)=0$. On the other hand, there are the points $\omega\in \mathbb{C}$ such that $\omega=\omega_1\frac{|k|}{|k_1|}$ where $D_{k_1}(\omega_1)=0$. The multiplicity of a pole belonging to the two families is smaller than or equal to $n_{k_1,\omega_1}+n_{k_2,\omega_2}+1$. Else the multiplicity of a pole of the first kind is smaller than or equal to $n_{k_1,\omega_1}+n_{k_2,\omega_2}-1$ and the multiplicity of a pole of the second kind is smaller than or equal to $n_{k_1,\omega_1}+1$.

\noindent $\bullet$ Else $\mathcal{N}^1_{k}$ is an entire function and the poles of $\mathcal{N}^2_{k}$ are the points $\omega\in \mathbb{C}$ such that $\omega=\omega_1+\omega_2$ where $D_{k_1}(\omega_1)=D_{k_2}(\omega_2)=0$ and its multiplicity is smaller than  or equal to $n_{k_1,\omega_1}+n_{k_2,\omega_2}-1$.
\end{lemma}


Now, we focus on proving Lemma \ref{lem_belle_N_is_nice}. However, using  analytic representation Theorem \ref{theo_anal_rep}, it is directly deduced of the two following lemmas (Lemma \ref{lem_F1_is_cool} and Lemma \ref{lem_F2_is_cool}) involving properties of $F^1_k$ and $F^2_k$.


\begin{lemma}
\label{lem_F1_is_cool}
For all $\lambda \in \mathbb{R}$ there exist $\beta \in (0,\frac{\pi}2)$ and an analytic and bounded function on $\Sigma_{\beta}$ denoted $r$ such that

\noindent $\bullet$ if $k_2 = -\gamma k_1, \ \gamma\in (0,1)$, then  for all $t>0$
\begin{equation}
\label{F^1_res}
 F^1_k(t) = \sum_{\substack{ D_{k_1}(\omega_1)=0 \\
													\Im \omega_1 \geq \lambda	}} R_{\omega_1}(t) e^{-i \omega_1 \frac{|k|}{|k_1|}t} + e^{\lambda t} r(t), 
\end{equation}
													with $R_{\omega_1}$ a polynomial of degree smaller than or equal to $n_{k_1,\omega_1}$,
													
\noindent $\bullet$ else, for all $t>0$
\begin{equation}
\label{F^1_nonres}
F^1_k(t) = e^{\lambda t} r(t). 
\end{equation}

\end{lemma}

\begin{lemma}
\label{lem_F2_is_cool}
For all $\lambda \in \mathbb{R}$ there exist $\beta \in (0,\frac{\pi}2)$ and an analytic and bounded function on $\Sigma_{\beta}$ denoted $r$ such that

\noindent $\bullet$ if $k_2 = -\gamma k_1, \ \gamma\in (0,1)$, then  for all $t>0$
\begin{equation}
\label{F^2_res}
F^2_k(t) = \sum_{\substack{ D_{k_1}(\omega_1)=0 \\
													\Im \omega_1 \geq \lambda	}} \left[ R_{k_1,k_2}^{\omega_1}(t) e^{-i \omega_1 \frac{|k|}{|k_1|}t} +  \sum_{\substack{ D_{k_2}(\omega_2)=0 \\
													\Im \omega_2 \geq \lambda	}} Q^{k_1,k_2}_{\omega_1,\omega_2}(t) e^{-i (\omega_1+\omega_2) t} 
													\right] + e^{\lambda t} r(t), 
\end{equation}
	with $Q^{k_1,k_2}_{\omega_1,\omega_2}$ a polynomial of degree smaller than or equal to $n_{k_1,\omega_1}+n_{k_2,\omega_2}-2$ and $R_{k_1,k_2}^{\omega_1}$  a polynomial of degree smaller than or equal to $n_{k_1,\omega_1}$ (if there exist $\omega_1,\omega_2$ such that $\omega_1+\omega_2 = \omega_1 \frac{|k|}{|k_1|}$ 
	 the maximal possible degree of $Q^{k_1,k_2}_{\omega_1,\omega_2}$ and $R_{k_1,k_2}^{\omega_1}$
	  is $n_{k_1,\omega_1}+n_{k_2,\omega_2} $  ) 
	
\noindent $\bullet$ else, for all $t>0$
\begin{equation}
\label{F^2_nonres}
F^2_k(t) = \sum_{\substack{ D_{k_1}(\omega_1)=0 \\
													\Im \omega_1 \geq \lambda	}}  \sum_{\substack{ D_{k_2}(\omega_2)=0 \\
													\Im \omega_2 \geq \lambda	}} Q^{k_1,k_2}_{\omega_1,\omega_2}(t) e^{-i (\omega_1+\omega_2) t} + e^{\lambda t} r(t),
\end{equation}
with $Q^{k_1,k_2}_{\omega_1,\omega_2}$ a polynomial of degree smaller than or equal to $n_{k_1,\omega_1}+n_{k_2,\omega_2}-2$.
\end{lemma}

\medskip

\begin{remark}
In Lemma \ref{lem_belle_N_is_nice}, we need that the inequality is true for a given $\tilde{\lambda}$, in order to apply Lemma \ref{lemma_key}. However, applying Theorem \ref{theo_anal_rep}, we deduce from Lemmae \ref{lem_F1_is_cool}
and \ref{lem_F2_is_cool} that the inequality is true \emph{for all} $\tilde{\lambda}\in \mathbb{R}$. On the other hand, we have needed that Lemma  \ref{lem_F1_is_cool} and  \ref{lem_F2_is_cool} are true for all $\lambda \in \mathbb{R}$, in order to prove that $\mathcal{N}_k^1$ and $\mathcal{N}_k^2$ are meromorphic.
\end{remark}

We are going to prove these lemmas distinguishing the non resonant case from the resonant case (when there exists $\gamma\in (0,1)$ such that $k_2=-\gamma k_1$). In order to get proofs as clear as possible we do not prove that the remainder term can be extended on complex cones and we only control them on $\mathbb{R}_+^*$. Indeed, there are no real issues to extend them and the arguments to control them on $\Sigma_{\alpha}$ or $\mathbb{R}_+^*$ are the same. Furthermore, the notations induced for the complex extensions are quite heavy and so do not help to understand the ideas. However, in the first proof, to give an example, we prove the analytic extension and we really estimate it.

\subsection{Proof of Lemma \ref{lem_F1_is_cool} in the non-resonant case.}  
Since we are studying the non-resonant case, there exists $\delta >0$ such that
\[  \forall \theta\in (0,1), \ \delta \leq |(1-\theta) k + \theta k_2|.  \]
Indeed, in the resonant case there exists $\gamma\in (0,1)$ such that $k_2=-\gamma k_1=\gamma(k_2-k)$, so that $(1-\gamma)k_2+\gamma k =0$.
We have proven in Proposition \ref{resol_rel_disp_lin} that there exists $\alpha\in (0,\frac{\pi}2)$ such that $\widehat{\psi}(t,k_1)$ extends to an analytic function on $\Sigma_{\alpha}$ and that there exists $\lambda_0\in \mathbb{R}$ and $M >0$ such that 
\[\forall z\in \Sigma_{\alpha},  |\widehat{\psi}(z,k_1)| \leq M e^{\lambda_0 \Re z}.\]
Furthermore, since $v\mapsto\widehat{g_0}(k_2,v)\in \mathscr{E}(\mathbb{R}^d)$, its Fourier transform extends to an entire function on $\mathbb{C}^d$ and we can assume (choosing $\alpha$ small enough) that for all $\lambda_2 \in \mathbb{R}$ there exists $C_{\lambda_2}>0$ such that
\begin{equation}
\label{lambda_2_control}
\forall z \in \Sigma_{\alpha} \mathbb{R}^d, \ |\Four  \left[ \widehat{g_0}(k_2,v)\right] (z)|\leq C_{\lambda_2} e^{\lambda_2 |\Re z|}. 
\end{equation}

Now observe that by a change of variable, $F^1_k(t)$ can be written as
\[  F^1_k(t) = t^2 \int_{0}^1 (1-\theta) \widehat{\psi}(\theta t,k_1)   \Four [\widehat{g_0}(k_2,v)](t \left( (1-\theta) k + \theta k_2\right) )  {\rm d\theta}. \]
Consequently, $F^1_k(t)$ naturally extends to an analytic function on $\Sigma_{\alpha}$. Now, we have to control $F^1_k(z) e^{-\lambda z}$ on $\Sigma_{\alpha}$ for any $\lambda \in \mathbb{R}$. Indeed, we have, for $z\in \Sigma_{\alpha}$, as we can assume $\lambda_2\le 0$,
\begin{align*}
|F^1_k(z) e^{-\lambda z}| &\leq  |z|^2e^{-\lambda \Re z}\int_{0}^{1}   |\widehat{\psi}(\theta z,k_1)|  (1-\theta) \left| \Four [\widehat{g_0}(k_2,v)](z \left((1-\theta) k + \theta k_2\right) ) \right| {\rm d\theta} \\
&\leq C_{\lambda_2} M |z|^2 \int_{0}^{1}   e^{(\lambda_0 \theta - \lambda) \Re z}   (1-\theta) e^{\lambda_2 \Re z |(1-\theta) k + \theta k_2| }{\rm d\theta} \\
&\leq C_{\lambda_2} M |z|^2   e^{(|\lambda_0| - \lambda + \delta \lambda_2) \Re z} \\
&\leq C_{\lambda_2} M \left(\frac{\Re z}{\cos \alpha}\right)^2   e^{(|\lambda_0| - \lambda + \delta \lambda_2) \Re z}.
\end{align*}
So this quantity is bounded uniformly with respect to $z\in  \Sigma_{\alpha}$ if $\lambda_2 < \frac{\lambda- |\lambda_0|}{\delta}.$

\subsection{Proof of Lemma \ref{lem_F1_is_cool} in the resonant case}
As explained before, from now, we do not pay attention to the analytic extension anymore. First, we use the resonance to give a more adapted expression of $F^1_k$
\begin{align*}
F^1_k(t) &= \int_{0}^t   \widehat{\psi}(\tau,k_1)   (t-\tau) \Four [\widehat{g_0}(k_2,v)](k_1 \left( (1-\gamma) t -   \tau \right) ) {\rm d\tau} \\
			&= \int_{-\gamma t}^{(1-\gamma) t} \widehat{\psi}( (1-\gamma) t -   s ,k_1) (\gamma t + s)  \Four \widehat{g_0}(k_2,k_1 s ) \textrm{d}s,
\end{align*}
making the change of variable $s\leftarrow (1-\gamma) t -   \tau $.
We want to expand $\psi$, so we introduce the dependency of $F^1_k$ with respect to $t\mapsto\widehat{\psi}( t ,k_1) $ by denoting $F^1_k[\widehat{\psi}(t,k_1)](t)$. Consequently, using the expansion of $\psi$ of Proposition \ref{resol_rel_disp_lin}, for any $\lambda_1 \in \mathbb{R}$, we get
\[ F^1_k[\widehat{\psi}(t,k_1)](t) = \sum_{\substack{ D_{k_1}(\omega_1)=0 \\
													\Im \omega_1 \geq \lambda_1	}}  F^1_k[  P_{k_1,\omega_1}(t) e^{-i\omega_1 t} ](t)+ F^1_k[   e^{\lambda_1 t} r_{k_1,\lambda_1}(t)](t),\]
where $r_{k_1,\lambda_1}$ is a bounded function on $\mathbb{R}_+^*$.

\medskip

First, we are going to control the remainder term $F^1_k[ e^{\lambda_1 t} r_{k_1,\lambda_1}(t)](t)$. Using the same control of the Fourier transform as previously (see \eqref{lambda_2_control}), we have,
as we can assume $\lambda_1,\lambda_2\le 0$,
\begin{multline*}
e^{-\lambda t} |F^1_k[  e^{\lambda_1 t} r_{k_1,\lambda_1}(t)](t) |\\
 \leq \| r_{k_1,\lambda_1} \|_{L^{\infty}} e^{-\lambda t} \int_{-\gamma t}^{(1-\gamma) t} e^{\lambda_1 [(1-\gamma) t -   s]}    (\gamma t + s)  \left| \Four [\widehat{g_0}(k_2,v)](k_1 s ) \right| \textrm{d}s \\
  \leq 	\| r_{k_1,\lambda_1} \|_{L^{\infty}} C_{\lambda_2} e^{-\lambda t} \int_{-\gamma t}^{(1-\gamma) t} e^{\lambda_1 [(1-\gamma) t -   s]}    (\gamma t + s) e^{ \lambda_2 |k_1| |s|} \textrm{d}s \\
 \leq 	\| r_{k_1,\lambda_1} \|_{L^{\infty}} C_{\lambda_2} e^{[(1-\gamma)\lambda_1-\lambda] t} \int_{\mathbb{R}}   (\gamma t + s) e^{ (\lambda_2 |k_1| - \lambda_1) |s|} \textrm{d}s.
 \end{multline*}
So this quantity is bounded uniformly with respect to $t\in \mathbb{R}_+^*$ if $(1-\gamma)\lambda_1<\lambda$ and $\lambda_2 |k_1|<\lambda_1$.

%

\medskip

Now, we are going to study one leading term of the type $F^1_k[ t^n e^{-i\omega_1 t} ](t)$. So, we are doing a new expansion.
\begin{align*}
F^1_k[ t^n e^{-i\omega_1 t} ](t) &= e^{-i \omega_1(1-\gamma) t } \int_{-\gamma t}^{(1-\gamma) t} \! \! \! \!  ( (1-\gamma) t -   s  )^n e^{i \omega_1  s } (\gamma t + s)  \Four [\widehat{g_0}(k_2,v)](k_1 s ) \textrm{d}s \\
											&=  \sum_{j=0}^{n+1} b_j t^j e^{-i \omega_1(1-\gamma) t }  \int_{-\gamma t}^{(1-\gamma) t} s^{n+1-j} e^{i \omega_1  s }   \Four [\widehat{g_0}(k_2,v)](k_1 s ) \textrm{d}s,
\end{align*}
where $b_0,\dots,b_{n+1}$ are real numbers. Here we recognise the leading terms of \eqref{F^1_res} since, by construction, $1-\gamma = \frac{|k|}{|k_1|}$.

\medskip

Then, observe that since the right integral is convergent (see \eqref{lambda_2_control}), there exists $A\in \mathbb{C}$ such that for any $t\in \mathbb{R}_+^*$, we have
\begin{equation*}
\begin{split}
 \int_{-\gamma t}^{(1-\gamma) t} \! \! s^{n+1-j} e^{i \omega_1  s }   \Four [\widehat{g_0}(k_2,v)](k_1 s ) \textrm{d}s =& A - \int_{-\infty}^{-\gamma t} \! \! \! s^{n+1-j} e^{i \omega_1  s }   \Four [\widehat{g_0}(k_2,v)](k_1 s ) \textrm{d}s \\ 
 &- \int_{(1-\gamma) t}^{+\infty} \! \!  \! s^{n+1-j} e^{i \omega_1  s }   \Four [\widehat{g_0}(k_2,v)](k_1 s ) \textrm{d}s  .
 \end{split}
\end{equation*}
The complex number $A$ is the leading term of this integral whereas the other ones are remainder terms. So we just have to control them. Indeed, we have
\begin{multline*}
  \left| e^{-\lambda t} t^j e^{-i \omega_1(1-\gamma) t } \int_{-\infty}^{-\gamma t} s^{n+1-j} e^{i \omega_1  s }   \Four [\widehat{g_0}(k_2,v)](k_1 s ) \textrm{d}s \right|\\
  \leq C_{\lambda_2} e^{-\lambda t} t^j e^{ \Im \omega_1(1-\gamma) t } \int_{\gamma t}^{\infty} s^{n+1-j} e^{\Im \omega_1  s }   e^{\lambda_2 |k_1| s} \textrm{d}s \\
 \leq C_{\lambda_2}  \int_{0}^{\infty}  e^{-\lambda \frac{s}\gamma} \left(\frac{s}\gamma\right)^j e^{ \Im \omega_1(1-\gamma) \frac{s}\gamma } s^{n+1-j} e^{\Im \omega_1  s }   e^{\lambda_2 |k_1| s} \textrm{d}s, 
 \end{multline*}
as we can assume $\lambda\le 0$ and since $t\le \frac{s}{\gamma}$.
Consequently, it is bounded uniformly with respect to $t$ if $|k_1|\gamma \lambda_2 <  \lambda - \Im \omega_1$.
The estimation of the third integral can be realized with the same ideas.

As we have a term in $t^{n+1}$, we see that $R_{\omega_1}$ is of degree $\le n_{k_1,\omega_1}$, since $P_{k_1,\omega_1}$ is of degree $\le n_{k_1,\omega_1}-1$.

\subsection{Proof of Lemma \ref{lem_F2_is_cool} in the non resonant case}
First, operating the change of variable $\tau'= t - \tau$, $s'= t-s$, we can write $F^2_k$ as
\[  F^2_k(t) = \int_{0\leq \tau'\leq s' \leq t}    \widehat{\psi}(t-\tau',k_1)    \widehat{\psi}(t-s',k_2)   \tau'    \Four \left[ k_2\cdot\nabla_v f^{eq} \right](k \tau'  + k_2 (s'-\tau')) \textrm{d}s' {\rm d\tau'}  ,\]
since, if $0\le s\le \tau \le t$ we get $0\le t-\tau\le t-s$ and $t-\tau\le t-s \le t$, that is $0\le  \tau' \le s'\le t$.
In order to get notations general enough but compact, we denote $u =  \Four \left[ k_2.\nabla_v f^{eq} \right]$. Since $u\in \mathscr{E}(\mathbb{R}^d)$, for all $\lambda_3 \in \mathbb{R}$ there exists a constant $C_{\lambda_3}>0$ such that
\begin{equation}
\label{lambda_3_control}
\forall \xi \in \mathbb{R}^d, \forall t>0, \ |u(t\xi)| \leq C_{\lambda_3} e^{\lambda_3 t |\xi|}  .
\end{equation}

\medskip

We define, for continuous functions $\phi_1,\phi_2$ with an exponential order, a bilinear operator $q$  by
\[  q[\phi_1,\phi_2](t) =  q[\phi_1(t),\phi_2(t)](t) =\int_{0\leq \tau\leq s \leq t}  \phi_1(t-\tau)   \phi_2(t-s)   \tau u(k \tau  + k_2 (s-\tau))  \textrm{d}s \ {\rm d\tau}.\]
With these notations, we have
\[ F^2_k(t) =  q[\widehat{\psi}(t,k_1)  ,\widehat{\psi}(t,k_2)](t).\]

\medskip

Consequently, using the expansions of $\widehat{\psi}(t,k_1) $ and $\widehat{\psi}(t,k_2)$ established in Proposition \ref{resol_rel_disp_lin}, we get\footnote{Realizing a decomposition of the form $$q[a_1+b_1,a_2+b_2] = q[a_1,a_2]+q[b_1,a_2+b_2]+q[a_1+b_1,b_2]-q[b_1,b_2].$$} for $\lambda_1,\lambda_2 \in \mathbb{R}$,
\begin{multline}
\label{dec_F_2}
F^2_k =    \sum_{\substack{ D_{k_1}(\omega_1)=0 \\
													\Im \omega_1 \geq \lambda_1	}}  \sum_{\substack{ D_{k_2}(\omega_2)=0 \\
													\Im \omega_2 \geq \lambda_2	}} q[ P_{k_1,\omega_1} e^{-i\omega_1 t}  , P_{k_2,\omega_2} e^{-i\omega_2 t}]   +  
													q[e^{\lambda_1 t} r_{k_1,\lambda_1}, \widehat{\psi}(t,k_2)]				\\+ q[\widehat{\psi}(tk_1),e^{\lambda_2 t} r_{k_2,\lambda_2} ]			
-q[e^{\lambda_1 t} r_{k_1,\lambda_1},e^{\lambda_2 t} r_{k_2,\lambda_2} ]
\end{multline}
where $r_{k_1,\lambda_1}$ and $r_{k_2,\lambda_2}$ are respectively bounded by constants $C_{\lambda_1}$ and $C_{\lambda_2}$.\\


Furthermore, we can also assume that there exists $\lambda_0 \in \mathbb{R}$ and $M>0$ such that
\[ \forall t >0, \  |\widehat{\psi}(t,k_1)| +| \widehat{\psi}(t,k_2) |\leq M e^{\lambda_0 t}.\]

\medskip

Finally, since we are treating the non-resonant case, we may assume that there exists $\delta>0$ such that
\begin{equation}
\label{carct_non_reson}
\forall 0\leq \tau \leq s, \  \delta s \leq  |\tau k + (s-\tau) k_2|  .
\end{equation}

\medskip

So first, we are going to control the remainder terms of \eqref{dec_F_2}. For example, we consider $q[e^{\lambda_1 t} r_{k_1,\lambda_1}, \widehat{\psi}(t,k_2)]$. So, if $t>0$, $\lambda_3<0$, $\lambda_1<0$, we have
\begin{multline*}
e^{-\lambda t} |q[e^{\lambda_1 t} r_{k_1,\lambda_1}, \widehat{\psi}(t,k_2)](t)| \\
\leq C_{\lambda_1}MC_{\lambda_3} e^{(\lambda_1+\lambda_0- \lambda) t}  \int_{0\leq \tau\leq s \leq t} e^{-\lambda_1 \tau - \lambda_0 s} \tau  e^{\lambda_3 |k \tau  + k_2 (s-\tau)|} \textrm{d}s {\rm d\tau} \\
																											\leq C_{\lambda_1}MC_{\lambda_3} e^{(\lambda_1+\lambda_0-\lambda) t}  \int_{0\leq \tau\leq s \leq t} e^{-\lambda_1 \tau - \lambda_0 s + \lambda_3 \delta s} \tau   \textrm{d}s {\rm d\tau} \\
																											\leq C_{\lambda_1}MC_{\lambda_3} t^2 e^{(\lambda_1+\lambda_0-\lambda) t}  \int_{s>0} e^{-\lambda_1 s - \lambda_0 s + \lambda_3 \delta s}  \textrm{d}s.
\end{multline*}
So, this quantity is bounded uniformly with respect to $t>0$ if $\lambda_1<\lambda - \lambda_0$ and $\lambda_3<\frac{\lambda_1 + \lambda_0}{\delta}<\frac{\lambda}{\delta}$.
Similarly, we could prove that if $\lambda_2$ is chosen negative enough then we could control $q[\widehat{\psi}(t,k_1),e^{\lambda_2 t} r_{k_2,\lambda_2} ](t)e^{-\lambda t}$ uniformly with respect to $t$,
and also $q[e^{\lambda_1 t} r_{k_1,\lambda_1},e^{\lambda_2t} r_{k_2,\lambda_2} ]$.

\medskip

$\!$Now, we consider a generic leading terms of \eqref{dec_F_2} of the type $q[ t^{n_1}e^{-i\omega_1 t},t^{n_2}e^{-i\omega_2 t}]$. So first, we can expand it
\begin{multline*}
q[ t^{n_1}e^{-i\omega_1 t}, t^{n_2}e^{-i\omega_2 t}](t) \\
=  \int_{0\leq \tau\leq s \leq t}  (t-\tau)^{n_1} e^{-i\omega_1 (t-\tau)}    (t-s)^{n_2} e^{-i\omega_2 (t-s)}     \tau u(k \tau  + k_2 (s-\tau))  \textrm{d}s {\rm d\tau} \\
= \sum_{j_1=0}^{n_1} \sum_{j_2=0}^{n_2} \big( b_{j_1,j_2} e^{-i (\omega_1+\omega_2) t} t^{n_1-j_1+n_2-j_2}  \\ \times \int_{0\leq \tau\leq s \leq t} \! \! \! \! \! \tau^{j_1+1}  s^{j_2} e^{i\omega_1 \tau + i \omega_2 s}      u(k \tau  + k_2 (s-\tau))  \textrm{d}s {\rm d\tau} \big),
\end{multline*}
where $b\in \mathbb{R}^{ \llbracket 0,n_1\rrbracket \times \llbracket 0,n_2\rrbracket }$ are some real coefficients.

\medskip

We observe that this last integral converge when $t$ goes to $+\infty$. Indeed, we have
\begin{align*}
 \left| \int_0^s  \tau^{j_1+1}  s^{j_2} e^{i\omega_1 \tau + i \omega_2 s}      u(k \tau  + k_2 (s-\tau))  {\rm d\tau} \right| &\leq C_{\lambda_3} s^{j_1+j_2+2} e^{(|\omega_1|+|\omega_2|)s} e^{ \lambda_3 \delta s}\\
 & \in L^1(\mathbb{R}_+), \ \textrm{ if } \delta \lambda_3 < -|\omega_1|-|\omega_2|. 
\end{align*}

\medskip

Consequently, there exists a complex constant $A\in \mathbb{C}$ such that
\begin{multline*}
 \int_{0\leq \tau\leq s \leq t}  \tau^{j_1+1}  s^{j_2} e^{i\omega_1 \tau + i \omega_2 s}      u(k \tau  + k_2 (s-\tau))  \textrm{d}s {\rm d\tau} \\
 = A -  \int_{\substack{ 0\leq \tau \leq s \\
																																t\leq s } } \tau^{j_1+1}  s^{j_2} e^{i\omega_1 \tau + i \omega_2 s}      u(k \tau  + k_2 (s-\tau))  \textrm{d}s {\rm d\tau}. 
\end{multline*}																							
This complex number $A$ generates the term of frequency $\omega_1+\omega_2$ in \eqref{F^2_nonres}. So we just need to prove that the other term is a remainder term controlling it. Indeed, we have
\begin{equation*}
\begin{split}
& e^{-\lambda t} \left| e^{-i (\omega_1+\omega_2) t} t^{n_1-j_1+n_2-j_2} \int_{\substack{ 0\leq \tau \leq s \\
					t\leq s } } \tau^{j_1+1}  s^{j_2} e^{i\omega_1 \tau + i \omega_2 s}      u(k \tau  + k_2 (s-\tau))  \textrm{d}s {\rm d\tau} \right| \\
\leq& C_{\lambda_3} e^{-\lambda t} e^{\Im (\omega_1+\omega_2) t} t^{n_1-j_1+n_2-j_2} \int_{\substack{ 0\leq \tau \leq s \\
					t\leq s } }  s^{j_1+j_2+1} e^{- \Im \omega_1 \tau - \Im \omega_2 s}   e^{\lambda_3 \delta s}   \textrm{d}s {\rm d\tau} \\
\leq & C_{\lambda_3} e^{-\lambda t} e^{\Im (\omega_1+\omega_2) t} t^{n_1-j_1+n_2-j_2} \int_{t \leq s }  s^{j_1+j_2+2} e^{|\Im \omega_1| s - \Im \omega_2 s}   e^{\lambda_3 \delta s}   \textrm{d}s\\
\leq & C_{\lambda_3}  \int_{s>0}  e^{ |\Im(\omega_1+\omega_2)|s-\lambda s} s^{n_1+2+n_2}  e^{ |\Im \omega_1| s - \Im \omega_2 s}   e^{\lambda_3 \delta s}   \textrm{d}s,
\end{split}
\end{equation*}
as $\lambda$ can be supposed $\le 0$, and this last quantity is finite if $\lambda_3$ is negative enough ($\lambda_3 < \frac{\lambda-|\Im(\omega_1+\omega_2)| -  |\Im \omega_1| +  \Im \omega_2}{\delta}$).

Concerning the degree, we see that it is $\le n_{k_1,\omega_1}-1+n_{k_2,\omega_2}-1$, since $n_1\le n_{k_1,\omega_1}-1$ and  $n_2\le n_{k_2,\omega_2}-1$, which corresponds to what is expected.

\subsection{Proof of Lemma \ref{lem_F2_is_cool} in the resonant case}
We consider now the last case, which is the most complex. We keep the notations of the previous subsection but we need a new expression of $q$ adapted to the resonance:
\begin{align*}
  q[\phi_1,\phi_2](t) &=  \int_{0\leq \tau\leq s \leq t}  \phi_1(t-\tau)   \phi_2(t-s)   \tau u(k_1 \left[ (1-\gamma) \tau  -\gamma (s-\tau) \right] )  \textrm{d}s {\rm d\tau},\\
  							&= \int_0^t \int_0^s  \phi_1(t-\tau)   \phi_2(t-s)   \tau u(k_1 \left[ \tau - \gamma s \right] ) {\rm d\tau} \textrm{d}s, \\
  							&= \int_0^t  \int_{-\gamma s}^{(1-\gamma) s}  \phi_1(t-\tau - \gamma s)   \phi_2(t-s) (\tau+\gamma s)  u(k_1 \tau) {\rm d\tau} \textrm{d}s.
\end{align*}
The term $\tau+\gamma s$ is quite heavy for our estimations, so we introduce a last notation
\[ q_{l}^m[\phi_1,\phi_2](t) =   \int_0^t  \int_{-\gamma s}^{(1-\gamma) s}  \phi_1(t-\tau - \gamma s)   \phi_2(t-s) \tau^l s^{m} u(k_1 \tau) {\rm d\tau} \textrm{d}s. \]
Consequently, we can expand $q[\phi_1,\phi_2]$ as follow
\[  q[\phi_1,\phi_2] = q_{1}^0[\phi_1,\phi_2] + \gamma q_{0}^1[\phi_1,\phi_2].  \]
We also introduce\footnote{Realizing a decomposition of the form: $q[a_1+b_1,a_2+b_2] = q[a_1,a_2]+q[a_1,b_2]+q[b_1,a_2+b_2].$} a new expansion of $F^2_k$ more adapted to the resonance
\begin{multline*}
F^2_k =     \sum_{\substack{ D_{k_1}(\omega_1)=0 \\
													\Im \omega_1 \geq \lambda_1	}}  \sum_{\substack{ D_{k_2}(\omega_2)=0 \\
													\Im \omega_2 \geq \lambda_2	}} q[ P_{k_1,\omega_1} e^{-i\omega_1 t}  , P_{k_2,\omega_2} e^{-i\omega_2 t}] \\
													+  \sum_{\substack{ D_{k_1}(\omega_1)=0 \\
													\Im \omega_1 \geq \lambda_1	}} q[ P_{k_1,\omega_1} e^{-i\omega_1 t}, e^{\lambda_2 t} r_{k_2,\lambda_2}] 
													+ q[e^{\lambda_1 t}r_{k_1,\lambda_1},\widehat{\psi}(t,k_2)].
\end{multline*}


Now, we are going to study each one of the terms of this expansion.

\paragraph{{\bf Last term}}
$ $

First, we control the last remainder term, $q[e^{\lambda_1 t}r_{k_1,\lambda_1},\widehat{\psi}(t,k_2)]$. Indeed, if $t>0$ we have
\begin{equation*}
\begin{split}
&e^{-\lambda t}|q_l^m[e^{\lambda_1 t}r_{k_1,\lambda_1},\widehat{\psi}(t,k_2)](z)|\\
\leq& C_{\lambda_1} M C_{\lambda_3} e^{-\lambda t} t^{l+m} \int_0^t  \int_{-\gamma s}^{(1-\gamma) s}  e^{\lambda_1[t-\tau - \gamma s]}   e^{\lambda_0(t-s)} e^{\lambda_3 |k_1| |\tau|} {\rm d\tau} \textrm{d}s \\
\leq& C_{\lambda_1} M C_{\lambda_3} \left(  \int_{\mathbb{R} }e^{-\lambda_1 \tau + \lambda_3 |k_1| |\tau|}   {\rm d\tau}  \right) t^{l+m} e^{ (-\lambda+\lambda_0 +\lambda_1)t} \left( \int_{0}^t  e^{ -\gamma s\lambda_1 -\lambda_0 s }   \textrm{d}s \right)\\
\leq& C_{\lambda_1} M C_{\lambda_3} \left(  \int_{\mathbb{R} }e^{-\lambda_1 \tau + \lambda_3 |k_1| |\tau|}   {\rm d\tau}  \right) t^{l+m+1} e^{ (-\lambda+\lambda_0 +\lambda_1)t}e^{ -\gamma t\lambda_1 -\lambda_0 t } \\
\leq& C_{\lambda_1} M C_{\lambda_3} \left(  \int_{\mathbb{R} }e^{\left(-\lambda_1 + \lambda_3|k_1|\right)   |\tau|}   {\rm d\tau}  \right) t^{l+m+1} e^{ (-\lambda+\lambda_1(1-\gamma))t}
\end{split}
\end{equation*}
So this last quantity is bounded uniformly with respect to $t$ if $\lambda_1$ and $\lambda_3$ are chosen negative enough. More precisely, we need $(1-\gamma)\lambda_1 < \lambda $ 
and $\lambda_3  |k_1| < \lambda_1$.

\paragraph{{\bf Second term}}
$ $

Now, we study the behavior of the second kind of term in the expansion of $F^2_k$. 

Expanding $P_{k_1,\omega_1} (t-\tau - \gamma s)$, we can write $q[ P_{k_1,\omega_1} e^{-i\omega_1 t}, e^{\lambda_2 t} r_{k_2,\lambda_2}](t)$ as a linear combination of term of the type  
$$t^{j} q_{l+1}^m[ e^{-i\omega_1 t}, e^{\lambda_2 t} r_{k_2,\lambda_2}](t)\ \textrm{and}\  t^{j} q_{l}^{m+1}[ e^{-i\omega_1 t}, e^{\lambda_2 t} r_{k_2,\lambda_2}](t),$$ with $j+l+m \leq \deg P_{k_1,\omega_1}$.

Let $t>0$, then we have
\begin{multline*}q_{l}^m[ e^{-i\omega_1 t}, e^{\lambda_2 t} r_{k_2,\lambda_2}](t) \\
= e^{-i\omega_1 t}  \int_0^t  \int_{-\gamma s}^{(1-\gamma) s}  e^{i\omega_1\tau} e^{i\omega_1\gamma s}    e^{\lambda_2 (t-s)} r_{k_2,\lambda_2}(t-s) \tau^l s^{m} u(k_1 \tau) {\rm d\tau} \textrm{d}s.
\end{multline*}
So, using \eqref{lambda_3_control},we introduce 
\[ \mathfrak{R}_{-}(s) =  \int_{-\infty}^{-\gamma s} e^{i\omega_1\tau} \tau^l u(k_1 \tau) {\rm d\tau} \textrm{ and }\mathfrak{R}_{+}(s) =  \int_{(1-\gamma)s}^{\infty} e^{i\omega_1\tau} \tau^l u(k_1 \tau) {\rm d\tau} \]
and 
\begin{equation}
A =  \int_{\mathbb{R}} e^{i\omega_1\tau} \tau^l u(k_1 \tau) {\rm d\tau} \textrm{ and } B_{\lambda_2,p} =  \int_0^{\infty} e^{-i\omega_1\gamma s}  s^p  e^{\lambda_2 s} r_{k_2,\lambda_2}(s) \textrm{d}s,
\label{Blam2p}
\end{equation}
where $B_{\lambda_2,p}$ is well defined if $\lambda_2$ is negative enough (i.e. $\lambda_2 < - \gamma |\lambda_0|$).
Consequently, we get (since $1-\gamma=\frac{|k|}{|k_1|}$)
\begin{equation*}
\begin{split}
&q_{l}^m[ e^{-i\omega_1 t}, e^{\lambda_2 t} r_{k_2,\lambda_2}](t) \\
=& A  e^{-i\omega_1 t}  \int_0^t e^{i\omega_1\gamma s}  s^m  e^{\lambda_2 (t-s)} r_{k_2,\lambda_2}(t-s) \textrm{d}s \\ 
&+ \int_0^t e^{i\omega_1\gamma s}    e^{\lambda_2 (t-s)} s^m r_{k_2,\lambda_2}(t-s) \left(  \mathfrak{R}_{-}(s)+ \mathfrak{R}_{+}(s)   \right) \textrm{d}s \\
=&A  e^{-i\omega_1 \frac{|k|}{|k_1|} t}  \int_0^t e^{-i\omega_1\gamma s}  (t-s)^m  e^{\lambda_2 s} r_{k_2,\lambda_2}(s) \textrm{d}s \\ 
&+ \int_0^t e^{i\omega_1\gamma s}    e^{\lambda_2 (t-s)} s^m r_{k_2,\lambda_2}(t-s)  \left(  \mathfrak{R}_{-}(s)+ \mathfrak{R}_{+}(s)   \right) \textrm{d}s\\
=& \sum_{p=0}^m C_m^p t^{m-p} AB_{\lambda_2,p} e^{-i\omega_1 \frac{|k|}{|k_1|} t}  \\
&+\sum_{p=0}^m C_m^p t^{m-p} A e^{-i\omega_1 \frac{|k|}{|k_1|} t}  \int_t^\infty e^{-i\omega_1\gamma s} s^p   e^{\lambda_2 s} r_{k_2,\lambda_2}(s) \textrm{d}s \\
&+ \int_0^t e^{i\omega_1\gamma s}    e^{\lambda_2 (t-s)}s^m r_{k_2,\lambda_2}(t-s)  \left(  \mathfrak{R}_{-}(s)+ \mathfrak{R}_{+}(s)   \right) \textrm{d}s,
\end{split}
\end{equation*}
where $C_m^p={{m}\choose{p}}$ is a binomial coefficient. Here there are three kinds of terms. The first one is one of expected leading term. The two others are remainder terms. So we have to control them.

\medskip

First, we control the second kind of term. If $t>0$ then
\begin{multline*}
 \left| e^{-\lambda t} e^{-i\omega_1 \frac{|k|}{|k_1|} t}  \int_t^\infty e^{-i\omega_1\gamma s} s^p   e^{\lambda_2 s} r_{k_2,\lambda_2}(s) \textrm{d}s \right| \\
\leq C_{\lambda_2} e^{-\lambda t} e^{\Im \omega_1 \frac{|k|}{|k_1|} t}  \int_t^\infty e^{\Im \omega_1\gamma s} s^p   e^{\lambda_2 s}  \textrm{d}s \\
\leq C_{\lambda_2} \int_{s>0} s^p e^{\left[ |\Im \omega_1| + |\lambda| + \lambda_2  \right]s} s^p   e^{\lambda_2 s}    \textrm{d}s.
\end{multline*}
So this last quantity is finite if $\lambda_2$ is negative enough.

\medskip

Then we control the last kind of term. If $t>0$ then
\begin{multline*}
\left| e^{-\lambda t}  \int_0^t e^{i\omega_1\gamma s}    e^{\lambda_2 (t-s)}s^m r_{k_2,\lambda_2}(t-s)   \mathfrak{R}_{-}(s) \textrm{d}s \right| \\
\leq C_{\lambda_2}C_{\lambda_3} e^{-\lambda t}  \int_0^t e^{ - \Im \omega_1 \gamma s}   e^{\lambda_2 (t-s)}s^m  \int_{\gamma s}^{\infty} e^{-\Im \omega_1\tau} \tau^l e^{\lambda_3 |k_1| \tau} {\rm d\tau} \textrm{d}s \\
\leq C_{\lambda_2}C_{\lambda_3} t^m e^{(\lambda_2-\lambda + |\Im \omega_1|\gamma) t} \int_{\tau >0}  \tau^l e^{\left( \lambda_3 |k_1| -\Im \omega_1 + \frac{|\lambda_2|}{\gamma} \right) \tau} {\rm d\tau}.
\end{multline*}
So this last quantity is bounded uniformly with respect to $t$ if $\lambda_2<\lambda - |\Im \omega_1|\gamma$ and $ \lambda_3 |k_1| < \Im \omega_1 )- \frac{|\lambda_2|}{\gamma} $. Of course, we could control the other remainder term (with $\mathfrak{R}_{+}$) in a similar way.

Concerning the degree, it is smaller or equal than the degree of $P_{k_1,\omega_1}$, that is $\le n_{k_1,\omega_1}-1$, as $j+m\le \deg P_{k_1,\omega_1}$. This is for the moment one degree less than what is expected in the Lemma \ref{lem_F2_is_cool}.

\begin{remark}
Note the term $B_{\lambda_2,p}$ in \eqref{Blam2p} is not explicit, as it relies on a remainder term of the first order dispersion relation. It is worth mentioning that this term contributes to the second order expansion, and not as a remainder term. 
\end{remark}

\paragraph{{\bf First term}}
$ $

Finally we study the first kind of terms in the expansion of $F^2_k$. These terms are of the type $q[ P_{k_1,\omega_1} e^{-i\omega_1 t}  , P_{k_2,\omega_2} e^{-i\omega_2 t}]$. By a straightforward calculation, as in the previous case, it can be extended as a linear combination of terms of the type $t^j q_{l+1}^m[ e^{-i\omega_1 t}  ,  t^{n} e^{-i\omega_2 t}]$ and $t^j q_{l}^{m+1}[ e^{-i\omega_1 t}  , t^n e^{-i\omega_2 t}]$ with $j+l+m=\deg P_{k_1,\omega_1}$ and $n\leq \deg P_{k_2,\omega_2}$.

\medskip

In order to pursue the proof for this first kind of terms, in the following elementary lemma, we introduce a useful algebraic decomposition. It is proven in Appendix \ref{Appendix2}.
\begin{lemma}
\label{lem_algb_dec}
 For all $n,m\in \mathbb{N}$, for all $\omega \in \mathbb{C}$, there exists $Q_{m,n,\omega},R_{m,n,\omega} \in \mathbb{C}[X]$ such that
\[ \forall t>0, \ \int_0^t e^{i\omega s} s^m (t-s)^n \textrm{d}s = Q_{m,n,\omega}(t) e^{i\omega t} + R_{m,n,\omega}(t)  .\]
If $\omega\neq 0$ then $\deg Q_{m,n,\omega} = m$ and $\deg R_{m,n,\omega} = n$. If $\omega = 0$ then $Q_{m,n,\omega}=0$ and $\deg R_{m,n,\omega} = m+n+1$. 
\end{lemma}

\begin{remark}
The fact that the degree of $R_{m,n,\omega}$ can change contains the discussion on the multiplicity. 
Indeed, it will be applied for $\omega=\gamma \omega_1+\omega_2$ which is equal to zero when $\omega_1+\omega_2= \frac{|k|}{|k_1|}\omega_1$, since $ \frac{|k|}{|k_1|}= \frac{|k_1+k_2|}{|k_1|}=(1-\gamma)$.
\end{remark}

\medskip

Furthermore, using the previous constructions, we introduce
\[ B(t) = \int_{s>0} e^{i\omega_1 \gamma s}  e^{i\omega_2 s} (t-s)^n s^m \left(  \mathfrak{R}_{-}(s)+ \mathfrak{R}_{+}(s)   \right)\textrm{d}s \in \mathbb{C}_n[t]. \]

\medskip

Now, if $t>0$, we have
\begin{equation*}
\begin{split}
&q_{l}^m[ e^{-i\omega_1 t}  , t^n e^{-i\omega_2 t}](t) \\
=&  \int_0^t  \int_{-\gamma s}^{(1-\gamma) s}  e^{-i\omega_1 (t-\tau - \gamma s)}  e^{-i\omega_2 (t-s)} (t-s)^n \tau^l s^{m} u(k_1 \tau) {\rm d\tau} \textrm{d}s \\
=& A  \int_0^t e^{-i\omega_1 (t - \gamma s)}  e^{-i\omega_2 (t-s)} (t-s)^n s^m \textrm{d}s \\
&+  \int_0^t e^{-i\omega_1 (t - \gamma s)}  e^{-i\omega_2 (t-s)} (t-s)^n s^m  \left(  \mathfrak{R}_{-}(s)+ \mathfrak{R}_{+}(s)   \right)\textrm{d}s\\
=& A  e^{-i (\omega_1+\omega_2)t } \left[ Q_{m,n,\gamma \omega_1 + \omega_2}(t) e^{ i (\gamma \omega_1 + \omega_2) t } + R_{m,n,\gamma \omega_1 + \omega_2}(t) \right] \\
&+  B(t) e^{-i (\omega_1+\omega_2)t } -  \int_t^\infty e^{-i\omega_1 (t - \gamma s)}  e^{-i\omega_2 (t-s)} s^m \left(  \mathfrak{R}_{-}(s)+ \mathfrak{R}_{+}(s)   \right)\textrm{d}s \\
=&  (A R_{m,n,\gamma \omega_1 + \omega_2} + B(t))  e^{-i (\omega_1+\omega_2)t } + A  Q_{m,n,\gamma \omega_1 + \omega_2}(t) e^{-i\omega_1 \frac{|k|}{|k_1|} t} \\
&-  \int_t^\infty e^{-i\omega_1 (t - \gamma s)}  e^{-i\omega_2 (t-s)} (t-s)^n s^m \left(  \mathfrak{R}_{-}(s)+ \mathfrak{R}_{+}(s)   \right)\textrm{d}s.
\end{split}
\end{equation*}

\medskip

Finally we just have to prove that this last integral is a remainder term. Indeed, we have
\begin{equation*}
\begin{split}
 & \left| e^{-\lambda t} \int_t^\infty e^{-i\omega_1 (t - \gamma s)}  e^{-i\omega_2 (t-s)} (t-s)^n s^m \mathfrak{R}_{-}(s) \textrm{d}s \right| \\
 \leq& C_{\lambda_3} t^n e^{(-\lambda + \Im \omega_1 + \Im \omega_2 )  t} \int_t^\infty e^{-(\gamma \Im \omega_1 + \Im \omega_2)s} s^m \int_{\gamma s}^{\infty}  e^{\Im \omega_1\tau} \tau^l e^{\lambda_3 |k_1| \tau}  {\rm d\tau}  \textrm{d}s \\
 \leq& C_{\lambda_3} \int_t^\infty e^{-\gamma s} \int_{\gamma s}^{\infty}  \frac{\tau^{n+l+m}}{\gamma^{n+m}} e^{(1+ \Im \omega_1+\frac{|-\lambda + \Im \omega_1 + \Im \omega_2|+ |\gamma \Im \omega_1 + \Im \omega_2|}{\gamma}+\lambda_3 |k_1|) \tau}  {\rm d\tau} \textrm{d}s \\
 \leq &C_{\lambda_3}  \int_{s>0} e^{-\gamma s} \textrm{d}s \int_{\tau>0}   \frac{\tau^{n+l+m}}{\gamma^{n+m}} e^{(1+ \Im \omega_1+\frac{|-\lambda + \Im \omega_1 + \Im \omega_2|+ |\gamma \Im \omega_1 + \Im \omega_2|}{\gamma}+\lambda_3 |k_1|) \tau}  {\rm d\tau} .
 \end{split}
\end{equation*}
So this last quantity is finite if $\lambda_3$ is negative enough.

Concerning the degree, we consider first the case $\gamma\omega_1+\omega_2\not=0$. As $B$ is of degree $\le n$ and $R_{m,n,\gamma\omega_1+\omega_2}$ is of degree $\le n$. So we get, as $j\le \deg P_{k_1,\omega_1}$
and $n\le \deg P_{k_2,\omega_2}$, that $Q^{k_1,k_2}_{\omega_1,\omega_2}$ is of degree $\le n_{k_1,\omega_1}-1+n_{k_2,\omega_2}-1$, which is the expected value. 
Now, as we can have a $q_l^{m+1}$ term, leading to $Q_{m+1,n,\gamma\omega_1+\omega_2}$ which is of degree $\le m+1$ and as $m$ can be chosen $\le \deg P_{k_1,\omega_1}$,
$R_{k_1,k_2}^{\omega_1}$ is of degree $\le n_{k_1,\omega_1}$, which is now the expected value.
We consider finally the case $\gamma\omega_1+\omega_2=0$, so that the terms $e^{-i(\omega_1+\omega_2)t}$ and $e^{-i\omega_1 \frac{|k|}{|k_1|} t}$ are the same. The terms of highest degree is then
$R_{m+1,n,\gamma \omega_1 + \omega_2}$ which is here of degree $\le m+n+2$, that is $\le n_{k_1,\omega_1}-1+n_{k_2,\omega_2}-1+2$. All the values found are thus those that are expected.



\subsection{\bf Proof of Proposition \ref{resol_rel_disp_o2}} 

\begin{proof}[Proof of Proposition \ref{resol_rel_disp_o2}]
In Proposition \ref{Dk_cool} and \ref{lem_belle_N_is_nice} we have proven that we can apply Lemma \ref{lemma_key} with $1-R(z)=D_k(z+i\widetilde{\lambda})$ and $N(z)=\mathcal{N}^1_{k}(z+i\widetilde{\lambda}) +  \mathcal{N}^2_{k}(z+i\widetilde{\lambda})$,
taking $\lambda_0=\widetilde{\lambda}/|k|$ in Proposition \ref{Dk_cool}: we get from Proposition  \ref{Dk_cool} and Lemma \ref{lem_belle_N_is_nice} 
$$
\forall z\in i\Sigma_{\gamma+\frac{\pi}{2}},\  |zR(z)|\le \frac{C}{|k|},\ \sup_{z \in i\Sigma_{\beta+\frac{\pi}2}} |zN(z)|<\infty. 
$$
 But the result of this lemma is that for all $\lambda\in \mathbb{R}$, we have
 \[ \frac{N(z)}{1-R(z)} = \Lap \bigg[   \sum_{\substack{  \omega \textrm{ pole of } \frac{N}{1-R} \\ 
													\Im \omega \geq \lambda	}} P_{\omega}(t) e^{-i \omega t} + e^{\lambda t} r(t) \bigg](z)      ,\]
with a function $r  \in \mathcal{H}(\Sigma_{\tilde{\beta}})$ analytic and bounded on $\Sigma_{\tilde{\beta}}$, for $\Im z$ large enough, with some $\tilde{\beta}$ satisfying $0<\tilde{\beta}<\gamma<\beta$ and $P_\omega$ is the polynomial such that
$$
   \frac{N(z)}{1-R(z)} \mathop{=}_{z\to \omega}  \Lap[P_\omega(t) e^{-i\omega t}] + \mathcal{O}(1).
$$

Thus, we have
\begin{align*}
\frac{\mathcal{N}^1_{k}(z)+\mathcal{N}^2_{k}(z)}{D_k(z)} &= \Lap \bigg[   \sum_{\substack{  \omega \textrm{ pole of } \frac{N}{1-R} \\ 
													\Im \omega \geq \lambda	}} P_{\omega}(t) e^{-i \omega t} + e^{\lambda t} r(t) \bigg](z-i\widetilde{\lambda})\\
&=\Lap \bigg[   \sum_{\substack{  \omega \textrm{ pole of } \frac{N}{1-R} \\ 
													\Im \omega \geq \lambda	}} P_{\omega}(t) e^{-i (\omega+i\widetilde{\lambda}) t} + e^{(\lambda+\widetilde{\lambda})t} r(t) \bigg](z)\\
&=\Lap \bigg[   \sum_{\substack{  \omega \textrm{ pole of } \frac{\mathcal{N}_k^1+\mathcal{N}_k^2}{D_k} \\ 
													\Im (\omega) \geq \lambda+\widetilde{\lambda}	}} P_{\omega-i\widetilde{\lambda}}(t) e^{-i \omega t} + e^{(\lambda+\widetilde{\lambda})t} r(t)\bigg](z)
\end{align*}

So, defining $\mu$ by
$$\widehat{\mu}(t,k)	= \sum_{\substack{ D(\omega)=1 \\
													\Im (\omega) \geq \lambda+\widetilde{\lambda}	}} P_{\omega-i\widetilde{\lambda}}(t) e^{-i \omega t} + e^{(\lambda+\widetilde{\lambda})t} r(t),
$$ we get \eqref{rel_disp_o2_nice}, which is \eqref{rel_disp_o2}. We finally have the expansion of Theorem \ref{The_th}.
Concerning the multiplicity, if one pole is common to $\mathcal{N}_k^1+\mathcal{N}_k^2$ and $D_k^{-1}$ we have to sum up the multiplicity, leading to add
$n_{k,\omega_1+\omega_2}-1$ to the range for $\ell$ and $n_{k,\frac{|k|}{|k_1|}\omega_1}-1$ to the range for $p$. The other concerns about the multiplicity follow from Lemmae \ref{lem_F1_is_cool} and \ref{lem_F2_is_cool}, and the condition $k\cdot k_1\not=0$ directly follows from the factor $k\cdot k_1$ in front of \eqref{defNk1} and \eqref{defNk2}. Note also that $\mathbb{R}_+^* \subset \Sigma_{\widetilde{\beta}}$, so that $r$ is bounded on $\mathbb{R}_+^*$
as stated in Theorem \ref{The_th}.
\end{proof}
 
\section{Numerical results}  \label{sec:5} Simulations have already been performed for multi-species and multi-dimensional simulations in \cite{us}, highlighting the relevance of second order expansion. We focus here more specifically on exhibiting a case where the Best frequency, that corresponds to the terms $B$ in Theorem \ref{The_th}, appears. 

\subsection{First example} We consider the one dimensional case ($d=1$ and $L_1=2\pi$) and solve numerically \eqref{VP_eq} with a Semi-Lagrangian scheme and an adapted $6$-th order splitting \cite{Casas2017}. 
$1D$ periodic centered Lagrange interpolation of degree $17$ is used in both $x$ and $v$ directions and the periodic Poisson solver is solved with fast Fourier transform.

Initial condition is
$
f_0(x,v) = f^{eq}(v) + \varepsilon g_0(x,v),$ with
$$
 \ f^{eq}(v) = e^{-v^2/2},\ g_0(x,v) = \cos(2x)e^{-v^2/(2\sigma_2^2)}+\cos(3x)e^{-v^2/(2\sigma_3^2)}
$$ and 
$\sigma_2 = 2^{1/4},\ \sigma_3 =\sqrt{\pi}/2$ and $\varepsilon=0.001$. 

We take $v\in [-v_{\max},v_{\max}]$, with $v_{\max}=10$. Numerical parameters are: the number of uniform cells in $x$ (resp. $v$) that are $N_x$ (resp. $N_v$) and the time step $\Delta t\in \mathbb{R}_+^*$,
leading to a grid which will be referred as $N_x\times N_v\times \Delta t$ grid.

The first Fourier mode $\widehat{E}_{1,num}(t)$ of the electric field $E:=-\nabla \Phi$ is computed from the simulation at each time step $t=t_n=n\Delta t$, using a discrete Fourier transform. 

We first compute the zeros of $D_k=D_{-k}$ (see Remark \ref{remDk}), for $|k|=1,2,3$ with greatest imaginary part that are
\begin{align*}
\omega_{ 1,\pm} &\simeq \pm2.511728081 -0.4796966410i,\\
 \omega_{ 2,\pm} &\simeq \pm3.734976684 -2.087460944i,\\
 \omega_{ 3,\pm} &\simeq \pm4.866872949 -4.113005968i.
\end{align*}
The second frequency of the  mode $1$ is $\omega_{ 1,\pm}^{(2)} \simeq  \pm3.498058625 -2.374303389i$.
Such zeros can be computing with a symbolic calculus software. An example using Maple is provided in the Appendix.
Here the modes that are initialized are $k_1,k_2\in\{\pm2,\pm3\}$. The main term is for $k=k_1+k_2=\pm1$, with $k_1=\mp2$ and $k_2=\pm3$, as $\omega_{\pm1}$ has the greatest imaginary part among the $\omega_{k_1+k_2}$, with
$k_1,k_2\in \{\pm2,\pm3\}$.
For having $k_2=-\gamma k_1$, with $\gamma\in (0,1)$, we have to take $k_2=\pm2$ and $k_1=\mp3$, so that the Best frequencies
$\omega_{b,\pm}$ of greatest imaginary part are defined by
$$
\omega_{b,\pm}=\frac{|k_1+k_2|}{|k_1|}\omega_{3,\pm} = \frac{\omega_{3,\pm}}{3}.
$$
In order to see such term, we have to remove the main part coming from $\omega_{\pm1}$. The procedure is detailed as follows.
From Theorem \ref{The_th},  we look here for
$$
 \Re(\widehat{E}_{1,num})(t) \simeq  \Re\left(ze^{-i\omega_1t}+(z_1+tz_2)e^{-i\omega_bt}\right), \textrm{with}\ z,z_1,z_2\in \mathbb{C},
$$
with $\omega_1=\omega_{1,+}$ or $\omega_1=\omega_{1,-}$, as it leads to the same value, and similarly for $\omega_b$.
We estimate $z$ by using a least square procedure: we first define
$$
\chi^2(y) = \sum_{t_{\min}\le t_j \le t_{\max}}\left(\Re\left(ye^{-i\omega_1 t_j}\right) -\Re(\widehat{E}_{1,num})(t_j)\right)^2
$$
and then define $z$ by minimizing this quantity, that is, $\chi^2(z) = \min_{y\in \mathbb{C}}\chi^2(y)$, which is explicitely given by as solution of 
$$
A^TA\left[\begin{array}{c}
\Re(z)\\
\Im(z)
\end{array}\right] = A^Tb,\ A=[\Re(e^{-i\omega_1t_j})_{j};-\Im(e^{-i\omega_1t_j})_{j}],\ b = \Re(\widehat{E}_{1,num})(t_j)_{j},
$$
with $A$ a matrix given by its 2 columns and $b$ a vector, all the three vectors being indexed by $j$ that goes through all the values such that $t_{\min}\le t_j \le t_{\max}$.

Once $z$ is found, we estimate $z_1$ and $z_2$ using again a least square procedure on the remainder: defining this time
$$
\tilde{\chi}^2(y_1,y_2) = \sum_{\tilde{t}_{\min}\le t_j \le \tilde{t}_{\max}}\left(\Re\left((y_1+t_jy_2)e^{-i\omega_b t_j}\right) -\Re\left(\widehat{E}_{1,num}(t_j)-ze^{-i\omega_1t}\right)\right)^2,
$$
$z_1$ and $z_2$ are obtained  by minimizing this quantity, that is, $$\tilde{\chi}^2(z_1,z_2) = \min_{y_1,y_2\in \mathbb{C}}\tilde{\chi}^2(y_1,y_2).$$
Again the solution is explicitely given, the matrix $A$ being here
$$
A=[\Re(e^{-i\omega_bt_j})_{j};-\Im(e^{-i\omega_bt_j})_{j};\Re(t_je^{-i\omega_bt_j})_{j};-\Im(t_je^{-i\omega_bt_j})_{j}].
$$

\begin{figure}
\includegraphics[width=\linewidth]{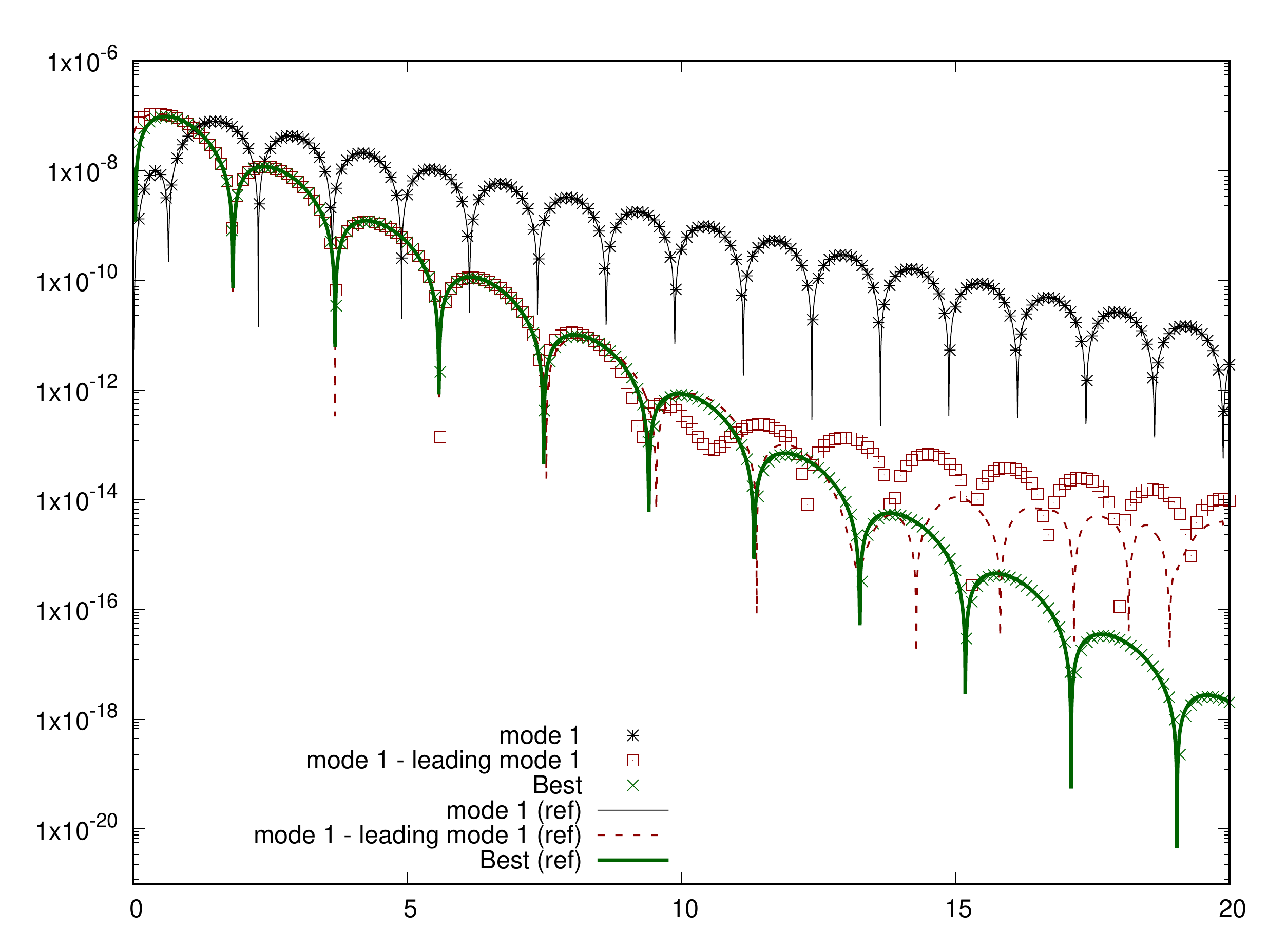}
\caption{Time evolution of $|\Re(\widehat{E}_{1,num})(t)|$ (mode 1), $|\Re\left(\widehat{E}_{1,num}(t)- ze^{-i\omega_1t}\right)|$ (mode 1 - leading mode 1) and 
$|\Re\left((z_1+t_jz_2)e^{-i\omega_b t_j}\right)|$ (Best), for coarse $128\times256\times0.1$ and refined $2048\times4096\times0.00625$ grids, the latter being referred as (ref) in the legend.
The parameters $[t_{\min},t_{\max}]=[17.5,35]$ and $[\tilde{t}_{\min},\tilde{t}_{\max}]=[1.75,17.5]$ are used for the least square procedures.}
\label{fignumbest}
\end{figure}

On Figure \ref{fignumbest}, we represent the time evolution of the real part of the first Fourier mode $|\Re(\widehat{E}_{1,num})(t)|$ in absolute value, 
together with $|\Re(\widehat{E}_{1,num})(t)- \Re\left(ze^{-i\omega_1t}\right)|$, that is the quantity where we have removed the main part
(it is a term $J$ in Theorem \ref{The_th}); the latter is compared to $|\Re\left((z_1+t_jz_2)e^{-i\omega_b t_j}\right)|$ that corresponds to the Best term. The parameters $t_{\min}$, $t_{\max}$, $\tilde{t}_{\min}$ and $\tilde{t}_{\max}$ are chosen properly so that, in the corresponding interval, the approximation is valid. Note that a too low value is not good, as the expansion is only asymptotic and we consider only one term which is the main term asymptotically. A too high value is also not good, as we have to face with the round off or numerical error and the nonlinear behavior (note that we do not solve here the second linearized equation but the full nonlinear equation). 
We observe a well agreement, which is even better, by refining the grid, so that we can claim that we have 
exhibited the Best frequency in the numerical results, which is fully coherent with the theoretical results.

\subsection{Another case where the Best frequency is almost dominant on a spatial mode}
Now we consider again $d=1$ (dimension $1$), but we change the spatial length of the domain $L_1 = 20\pi$, and take  
$$
 \ f^{eq}(v) = e^{-v^2/2},\ g_0(x,v) = \cos(x)e^{-v^2/(2\sigma_2^2)}+\cos(0.1x)e^{-v^2/(2\sigma_3^2)}
$$ and 
$\sigma_2 = 2^{1/4},\ \sigma_3 =\sqrt{\pi}/2$ and $\varepsilon=0.001$. Now the modes that are initialized are $k_1,k_2\in \{\pm 1,\pm0.1\}$. We now need to know (we already have the value of $\omega_{1,\pm}$ from the 
previous subsection)
\begin{align*}
\omega_{0.1,\pm} &\simeq \pm1.592755970+3.218848582\cdot10^{-52}i,\\
\omega_{0.2,\pm} &\simeq \pm1.621955006 -2.569883158\cdot10^{-12}i,\\
 \omega_{0.9,\pm} &\simeq \pm2.382548194-0.3594880484i,\\
 \omega_{1.1,\pm} &\simeq\pm2.639613224-0.6100786528i.
\end{align*}
The second frequency of the mode $0.9$ is $\omega_{0.9,\pm}^{(2)}\simeq\pm3.181466437 -2.102684847i$.
The possible values of  $k=k_1+k_2$ are in the set $\{\pm0.2,\pm0.9,\pm1.1,\pm2\}$. The first order expansion already gives a term that is not damped (the imaginary part is almost equal to zero).
We also have terms on the second order expansion that are not damped (for $k=\pm0.2$). Nevertheless, if one consider the mode $k=\pm0.9$, one can look at $|\Re(\widehat{E}_{0.9,num})(t)|$. From Theorem \ref{The_th},  
we look thus here for an approximation of  $\varepsilon^{-2}\Re(\widehat{E}_{0.9,num})(t)$ in the form
 $$
\mathcal{E}(t,z) =  \Re\left(z_1e^{-i\omega_{0.9}t}+(z_2t+z_3)e^{-i0.9\omega_1t}+z_4e^{-i(\omega_{1}+\omega_{0.1,-})t}+z_5e^{-i(\omega_{1}+\omega_{0.1,+})t}\right), 
$$
with $z=(z_1,z_2,z_3,z_4,z_5)\in \mathbb{C}^5$, using again $\omega_{\ell}=\omega_{\ell,+}$ or $\omega_{\ell}=\omega_{\ell,-}$, for $\ell\in \mathbb{R}$, as it leads to the same result.
In order to estimate $z$, we compute
$$
\min_{y\in \mathbb{C}^5} \sum_{t_{\min}\le t_j \le t_{\max}}\left(e^{\lambda t_j}\Re\left(\mathcal{E}(t_j,y)-\varepsilon^{-2}\widehat{E}_{0.9,num}(t_j)\right)\right)^2,
$$
that is attained for $y=z$, by using the least square method as previously. Note that we add here the weight $e^{\lambda t}$, with $\lambda=0.48$ and then we look for all the coefficients in one step.
The choice of the value of  $\lambda$ is coherent with the fact that from Theorem \ref{The_th}, the function $e^{\lambda t}\left(\varepsilon^{-2}\Re(\widehat{E}_{0.9,num})(t)-\mathcal{E}(t,z)\right)$ 
should be bounded. Numerical results are shown on Figure \ref{fignumbest2}. We use $t_{\min}=0$ and $t_{\max}=30$ for the coarse grid and  have increased $t_{\max}$ to $35$ for the fine grid
(for the fine grid, we could even increase this value, which was not possible for the coarse grid: the results were worse, as the solution is not precise enough for the coarse grid on late times, as shown on Figure \ref{fignumbest2}). 
For the fine grid, we could also not really increase further than around  $t_{\max}=50$, as we
are limited, with nonlinear effects, convergence and/or machine precision; we have also preferred not to go until $t_{\max}=50$, as it leads to a worser matching, since the least square procedure tends to match for 
values around $50$, where the matching is less good. We could also change the initial time, but it has not so much impact, as it was the case for the previous subsection, since we have added here a weight function in the 
least square procedure. 
We emphasize that we can again exhibit the Best frequency and also the two other types of frequencies, which are all in the same range, 
for this example. In order to get this results, we note that we had to adapt he strategy concerning the least square method that was presented for the first example; this is due to the fact the several modes 
are in a similar range, and  it was not easy to use the first procedure (used for the first example) to catch the different frequencies. 

\begin{figure}
\begin{center}
\includegraphics[width=0.8\linewidth]{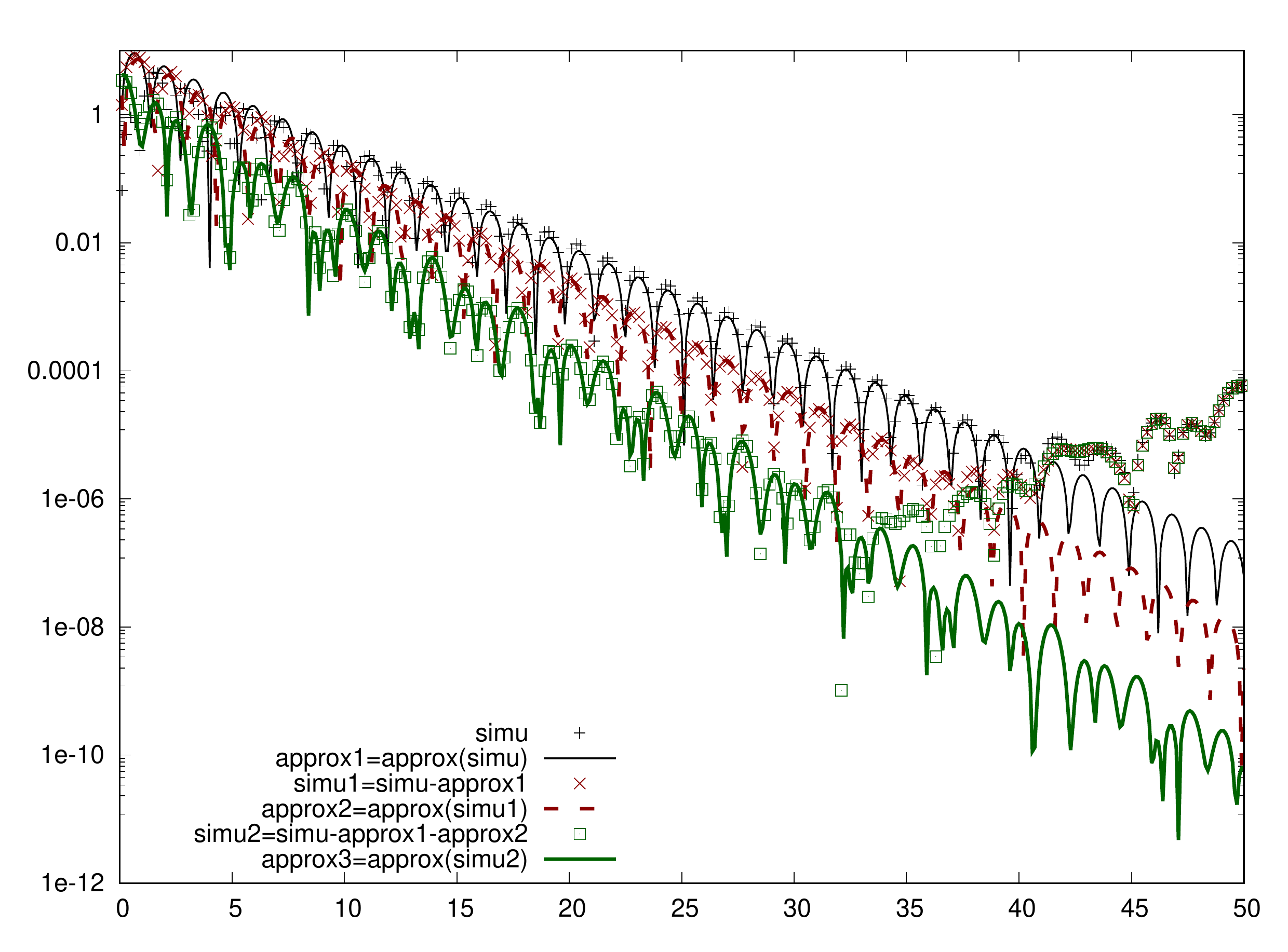}\\
\includegraphics[width=0.8\linewidth]{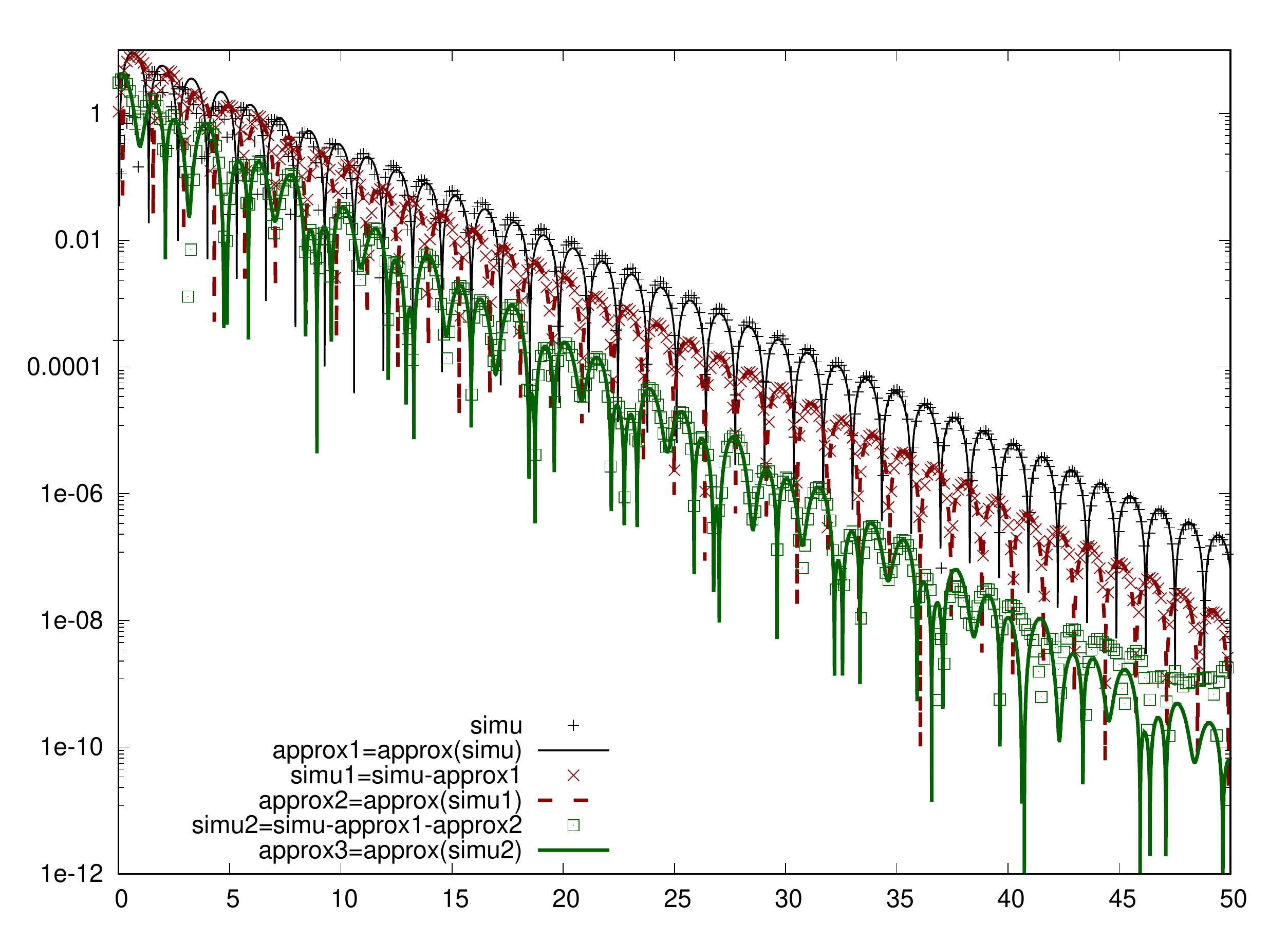}\\
\end{center}
\captionsetup{singlelinecheck=off}
\caption[foo bar]{
Time evolution of 
\begin{itemize}
\item $|\Re(\varepsilon^{-2}\widehat{E}_{1,num})(t)|$ (simu) 
vs $|z_1e^{-i\omega_{0.9}t}|$ (approx1),
\item $|\Re\left(\varepsilon^{-2}\widehat{E}_{1,num}(t)- z_1e^{-i\omega_{0.9}t}\right)|$ (simu1) 
\item[$ $] vs $|\Re\left((z_2t+z_3)e^{-i0.9\omega_1t}\right)|$ (approx2),
\item $|\Re\left(\varepsilon^{-2}\widehat{E}_{1,num}(t)- z_1e^{-i\omega_{0.9}t}-(z_2t+z_3)e^{-i0.9\omega_1t}\right)|$ 
\item[$ $] (simu2)  vs $|\Re\left(z_4e^{-i(\omega_{1}+\omega_{0.1,-})t}+z_5e^{-i(\omega_{1}+\omega_{0.1,+})t}\right)|$ (approx3),
\end{itemize}
for coarse (top) $128\times256\times0.1$ and refined (bottom) $2048\times4096\times0.00625$ grids.
The parameters for the least square procedure is $[t_{\min},t_{\max}]=[0,30]$ for the coarse grid and $[t_{\min},t_{\max}]=[0,35]$ for the refined grid.
}
\label{fignumbest2}
\end{figure}
The values of $z$ are given here for coarse and fine mesh:
\begin{align*}
z_{1,{\rm coarse}}&\simeq 1.2463 -11.578i,&\  z_{1,{\rm fine}}&\simeq 1.2183 -11.548i,\\
z_{2,{\rm coarse}}&\simeq 0.21502+0.28932i,&\  z_{2,{\rm fine}}&\simeq 0.23103+0.31652i,\\
z_{3,{\rm coarse}}&\simeq -4.2852+9.2615i,&\  z_{3,{\rm fine}}&\simeq -4.5484+8.9068i,\\
z_{4,{\rm coarse}}&\simeq 2.3853+1.2186i,&\  z_{4,{\rm fine}}&\simeq2.7369+1.1629i,\\
z_{5,{\rm coarse}}&\simeq 1.5556+1.2385i,&\  z_{5,{\rm fine}}&\simeq1.5611+1.1445i.
\end{align*}

\subsection{A $2D$ case}

\subsubsection{Looking for Best frequencies in 2D}

Finally, we focus on a $2D$ case. Here we can write $k=k_1+k_2$ with
$$
k_j= \left(m_j\frac{2\pi}{L_1},n_j\frac{2\pi}{L_2}\right),\ j=1,2,\ m_j,n_j\in \mathbb{Z}.
$$ 

Now if $k=\gamma k_1$, with $\gamma\in (0,1)$, we get:
$$
m_1+m_2=\gamma m_1,\ n_1+n_2=\gamma n_1, 
$$
which leads to
$$
1-\gamma = -\frac{m_2}{m_1} = -\frac{n_2}{n_1},
$$
if $m_1\not=0$ and $n_1\not=0$. If $m_1$ or $m_2=0$, we get $m_1=m_2=0$, and similarly for $n_1$ and $n_2$. In order to have a "real" $2D$ case, we can suppose that $m_1\not=0$ and $n_1\not=0$.
We have
$$
\frac{-m_2}{m_1} =  \frac{-n_2}{n_1}={1-\gamma} = \frac{p}{q},\ p,q\in \mathbb{N},\ p<q, \ p\wedge q =1.
$$
So we obtain $-m_2q = pm_1$, and thus $m_1=\ell q$, $\ell\in \mathbb{Z}^*$ and $-m_2=\ell p$ together with $n_1=\tilde{\ell} q$, $\tilde{\ell}\in \mathbb{Z}^*$ and $-n_2=\tilde{\ell} p$.
Note that we then have $k\cdot k_1 = \gamma |k_1|^2\not=0$.

\subsubsection{A $2D$ test case with Best frequency}

We choose here $L_1=L_2=L$, $m_1=n_1=3,\ m_2=n_2=-2$, so that 
$$
k_1 = (3,3)\frac{2\pi}{L},\ k_2 = (-2,-2)\frac{2\pi}{L},\ k_1+k_2=(1,1)\frac{2\pi}{L},
$$
and
$$
|k_1| = 3\sqrt{2}\frac{2\pi}{L},\ |k_2| = 2\sqrt{2}\frac{2\pi}{L},\ |k_1+k_2| = \sqrt{2}\frac{2\pi}{L}.
$$

We will write $\omega_{\ell}$, instead of $\omega_{\ell,+}$ or $\omega_{\ell,-}$, when we can either use $\omega_{\ell,+}$ or $\omega_{\ell,-}$. We will need for this subsection and the next one, the following values
(note that the values are here not the same as in the one dimensional case, since the dispersion relation is not the same, as we have considered here a normalized Maxwellian):

\begin{align*}
\omega_{\sqrt{2}/10,\pm} &\simeq \pm1.030839024 - 6.410202539\cdot10^{-10}i,\\
\omega_{2\sqrt{2}/10,\pm} &\simeq \pm1.140206800 - 0.007780445579i,\\
 \omega_{3\sqrt{2}/10,\pm} &\simeq \pm1.316627173 - 0.08467369148i,\\
 \omega_{\sqrt{5}/10,\pm} &\simeq\pm1.081943401 - 0.0004485284614i,\\
 \omega_{\sqrt{13}/10,\pm} &\simeq\pm1.234323666  - 0.04025247555i.
\end{align*}

Also, the second frequency of the mode $\sqrt{2}/10$ is $\omega_{\sqrt{2}/10,\pm}^{(2)}\simeq\pm0.5196579915 -0.2520173386i$.
%
%
%
%

The main frequencies that intervene on the spatial mode $(1,1)\frac{2\pi}{L}$ are $\omega_{\sqrt{2}\frac{2\pi}{L},\pm}$, and $\frac{\omega_{3\sqrt{2}\frac{2\pi}{L},\pm}}{3}$ (the last one is the Best frequency).
In that case, we expect a similar behavior as the test case of the first subsection.

We use here $f^{eq}(v) = \frac{1}{2\pi}e^{-v_1^2/2-v_2^2/2}$, with $x=(x_1,x_2),\ v=(v_1,v_2)$, together with
$$
g_0(x,v) = \left(\cos(0.3x_1+0.3x_2)+\cos(0.2x_1+0.2x_2)\right)f^{eq}(v),
$$
taking $L=20\pi$. We solve again numerically \eqref{VP_eq} with a Semi-Lagrangian scheme and an adapted $6$-th order splitting \cite{Casas2017} (here $d=2$).
The parameter $\varepsilon$ is always fixed to $\varepsilon=10^{-3}$.

We take $v_{1},v_2\in [-6,6]$, $32$ cells in $x_1$ and $x_2$ directions, $64$ cells in $v_1$ and $v_2$ directions; time step is fixed to $\Delta t=0.1$, leading to a
$32\times32\times64\times64\times0.1$  grid. The diagnostics are here obtained form the charge density $\rho(t,x_1,x_2) = \int_{\mathbb{R}^2} f{\rm d}v$ (computed from trapezoidal rule):
we define $\widehat{\rho}_{\ell_1,\ell_2,num}(t)$ the Discrete Fourier Transform of the charge density at time $t=t_n=n\Delta t$. Results are given on Figure \ref{fignumbest3}. 
The least square procedure is here applied to minimize:
$$
\min_{y\in \mathbb{C}^3} \sum_{t_{\min}\le t_j \le t_{\max}}\left(e^{\lambda t_j}\Re\left(\mathcal{E}(t_j,y)-\varepsilon^{-2}\widehat{\rho}_{1,1,num}(t_j)\right)\right)^2,
$$
with
 $$
\mathcal{E}(t,y) =  \Re\left(y_1e^{-i\omega_{\sqrt{2}/10}t}+(y_2t+y_3)e^{-i\frac{1}{3}\omega_{3\sqrt{2}/10}t}
\right), 
$$
and is attained for $y_j=z_j,\ j=1,2,3$, where the $z_j$ are given in Figure \ref{fignumbest3}.
We clearly see on Figure \ref{fignumbest3}, that the Best frequency $\frac{\omega_{3\sqrt{2}\frac{2\pi}{L}}}{3}$ is needed: 
with the combination of the main frequency $\omega_{\sqrt{2}\frac{2\pi}{L}}$
the simulated mode $\widehat{\rho}_{1,1,num}$ is accurately asymptotically described. In that case, we see both frequencies are useful; the main frequency is not enough as we can see it on Figure \ref{fignumbest3}.
Indeed both modes (main and Best) are shown (they are shifted towards bottom of the Figure in order to see them better), and we see that the \emph{combination} of the modes is needed to describe
the simulated mode.   
\subsubsection{A $2D$ test case without Best frequency}
Now, if we change and take $n_1=2,\ n_2=-1$, we have no more Best frequency, and the main  frequencies that intervene on the same spatial mode $(1,1)\frac{2\pi}{L}$ are
$\omega_{\sqrt{2}\frac{2\pi}{L},\pm}$ and $\omega_{\sqrt{13}\frac{2\pi}{L},\pm}+\omega_{\sqrt{5}\frac{2\pi}{L},\pm}$ (these frequencies were defined in the previous subsection), as we have this time
$$
k_1 = (3,2)\frac{2\pi}{L},\ k_2 = (-2,-1)\frac{2\pi}{L},\ k_1+k_2=(1,1)\frac{2\pi}{L},
$$
and
$$
|k_1| = \sqrt{13}\frac{2\pi}{L},\ |k_2| = \sqrt{5}\frac{2\pi}{L},\ |k_1+k_2| = \sqrt{2}\frac{2\pi}{L},
$$
the initial data being changed to
$$
g_0(x,v) = \left(\cos(0.3x_1+0.2x_2)+\cos(0.2x_1+0.1x_2)\right)f^{eq}(v),
$$
and we have still $L=20\pi$. For the least square procedure, we consider the minimization problem

$$
\min_{y\in \mathbb{C}^3} \sum_{t_{\min}\le t_j \le t_{\max}}\left(e^{\lambda t_j}\Re\left(\mathcal{E}(t_j,y)-\varepsilon^{-2}\widehat{\rho}_{1,1,num}(t_j)\right)\right)^2,
$$
with
 $$
\mathcal{E}(t,y) =  \Re\left(y_1e^{-i\omega_{\sqrt{2}/10}t}+y_2e^{-i(\omega_{\sqrt{5}/10,+}+\omega_{\sqrt{13}/10,-})t}+y_3e^{-i(\omega_{\sqrt{5}/10,+}+\omega_{\sqrt{13}/10,+})t}
\right), 
$$
attained for $y_j=z_j,\ j=1,2,3$, where the $z_j$ are given in Figure \ref{fignumbest4}. We remark here that we have some unexpected frequency at the beggining which might be interpreted as a Best frequency (the simu-first approx curve),
but such one is damped and we get the right asymptotic behavior, which shows that we cannot get a Best frequency in the asymptotic limit, which is fully consistant with Theorem \ref{The_th}.



\begin{figure}
\begin{center}
\includegraphics[width=0.9\linewidth]{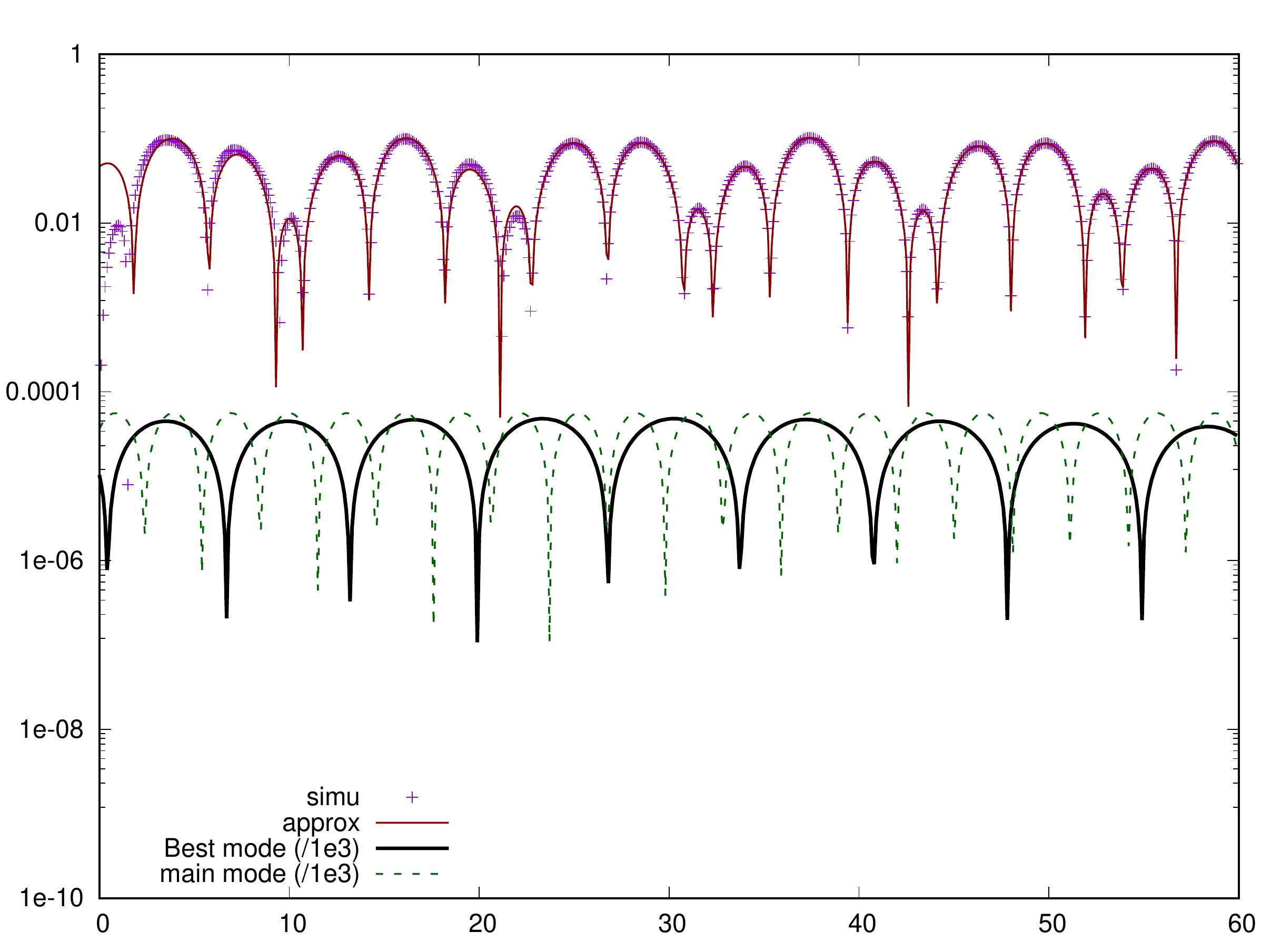}\\
\end{center}
\captionsetup{singlelinecheck=off}
\caption[foo bar]{
A $2D$-case with Best frequency: 
time evolution of 
\begin{itemize}
\item $|\Re(\varepsilon^{-2}\widehat{\rho}_{1,1,num})(t)|$ (simu) 
\item $|\Re\left(z_1e^{-i\omega_{\sqrt{2}/10}t}+(z_2t+z_3)e^{-i\frac{1}{3}\omega_{3\sqrt{2}/10}t}\right)|$ (approx) 
\item $10^{-3}|\Re\left(z_1e^{-i\omega_{\sqrt{2}/10}t}\right)|$ (main mode /1e3)
\item $10^{-3}|\Re\left((z_2t+z_3)e^{-i\frac{1}{3}\omega_{3\sqrt{2}/10}t}\right)|$ (Best mode /1e3)
\end{itemize}
}
The parameters $\lambda=0.09$ and $[t_{\min},t_{\max}]=[0,60]$ are used for the least square procedure to fit (simu) by (approx) and leads to $z_1\simeq 0.036159+0.042602i$, 
$z_2\simeq  -0.0031761-0.00089598i$ and $z_3\simeq 0.010351-0.046355i$.
\label{fignumbest3}
\end{figure}


\begin{figure}
\begin{center}
\includegraphics[width=0.9\linewidth]{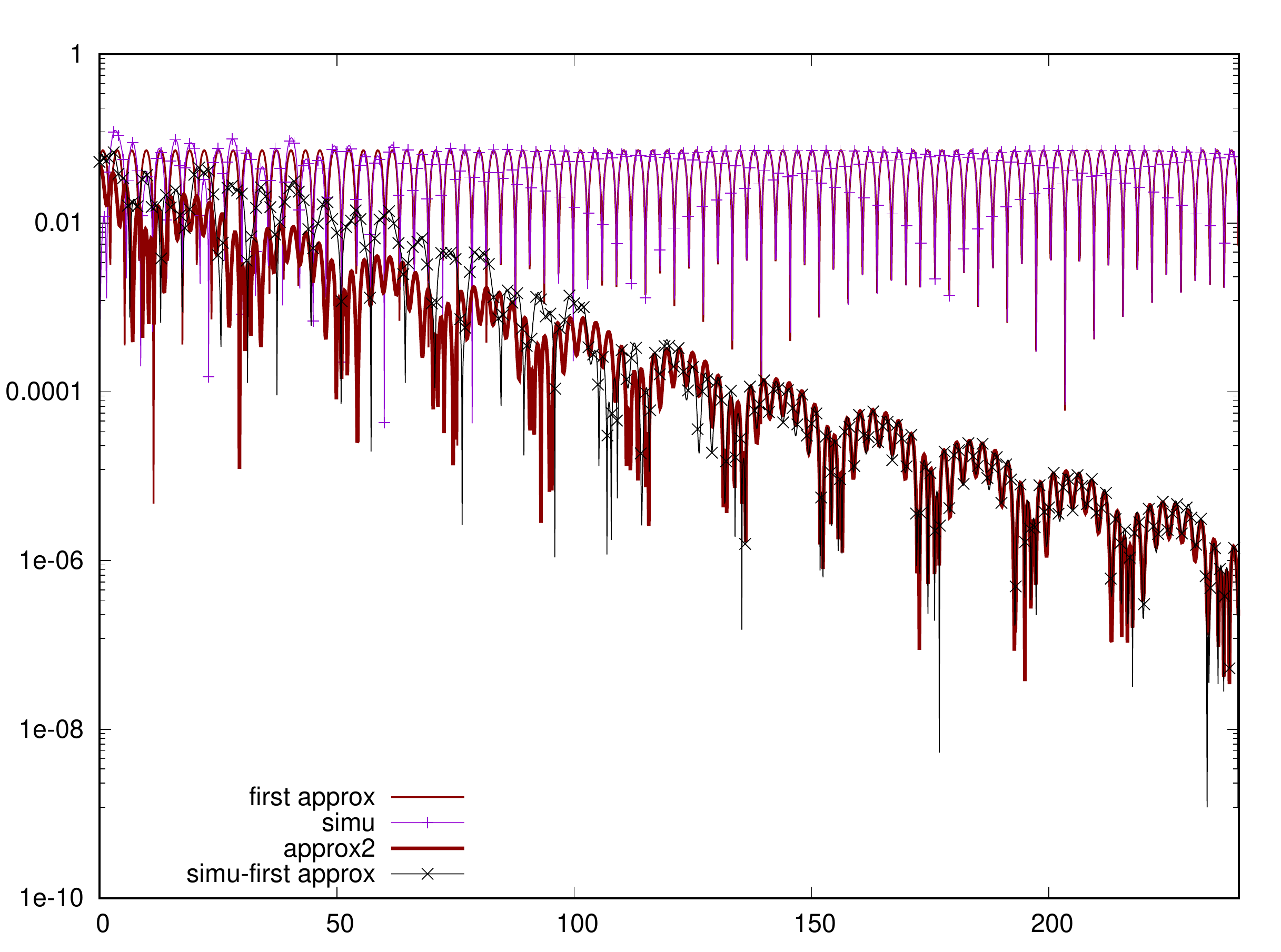}\\
\end{center}
\captionsetup{singlelinecheck=off}
\caption[foo bar]{
A $2D$-case without Best frequency: 
time evolution of 
\begin{itemize}
\item $|\Re(\varepsilon^{-2}\widehat{\rho}_{1,1,num})(t)|$ (simu) 
\item $|\Re\left(z_1e^{-i\omega_{\sqrt{2}/10}t}\right)|$ (first approx) 
\item $|\Re\left(\varepsilon^{-2}\widehat{\rho}_{1,1,num}(t)-z_1e^{-i\omega_{\sqrt{2}/10}t}\right)|$ (simu - first approx)
\item $|\Re\left(z_2e^{-i(\omega_{\sqrt{5}/10,+}+\omega_{\sqrt{13}/10,-})t}+z_3e^{-i(\omega_{\sqrt{5}/10,+}+\omega_{\sqrt{13}/10,+})t}\right)|$ (approx2)
\end{itemize}
}
The parameters $\lambda=0.05$ and $[t_{\min},t_{\max}]=[0,240]$ are used for the least square procedure leads to $z_1\simeq 0.052836+0.049810i$, 
$z_2\simeq  -0.032921-0.0010657i$ and $z_3\simeq -0.013703-0.0050901i$.
\label{fignumbest4}
\end{figure}

\section{Appendix}
\subsection{Some remarks about space $\mathscr{E}(\mathbb{R}^d)$}
\label{Appendix1} The aim of this subsection is to present some tools to construct explicit examples of functions of $\mathscr{E}(\mathbb{R}^d)$ (characterized by \eqref{def_Erd}).

\medskip

The {\it Gelfand-Shilov} spaces $S_{\alpha}^{\beta}(\mathbb{R}^d)$ provide many useful examples of functions of $\mathscr{E}(\mathbb{R}^d)$. They are defined, for $\alpha,\beta>0$, by
\begin{multline*}  S_{\alpha}^{\beta}(\mathbb{R}^d) := \{ f\in \mathscr{S}(\mathbb{R}^d) \ | \ \exists \varepsilon,C>0,\forall v,\xi\in \mathbb{R}^d, \ |f(v)|\leq Ce^{-\epsilon |v|^{\frac1{\alpha}}}\\ \textrm{ and } |\Four f(\xi)|\leq C e^{-\epsilon |\xi|^{\frac1{\beta}}}    \} .\end{multline*}
Many details about these spaces can be found in \cite{MR2668420}, in particular these spaces are stable by multiplication by a polynomial or a trigonometric polynomial, derivation and the natural action of the affine group of $\mathbb{R}^d$. Furthermore, we obviously have $\Four  S_{\alpha}^{\beta}(\mathbb{R}^d)= S_{\beta}^{\alpha}(\mathbb{R}^d)$.
\begin{proposition}
 If $\nu\in (0,1)$, then $S_{\nu}^{1-\nu}(\mathbb{R}^d)\subset \mathscr{E}(\mathbb{R}^d)$.
\end{proposition}
\begin{proof}
It is a direct corollary of Proposition $6.1.8$ of \cite{MR2668420}.
\end{proof}
\begin{example} $\empty$
\begin{itemize}
\item $|v|^2 \cos(v_1-v_2) e^{-v_1^2 - (v_1+v_2)^2}\in S_{\frac12}^{\frac12}(\mathbb{R}^2) \subset \mathscr{E}(\mathbb{R}^2)$,
\item If $k\in \mathbb{N}^*$ then $e^{-v^{2k}}\in S_{\frac{1}{2k}}^{{1-\frac1{2k}}}(\mathbb{R})\subset \mathscr{E}(\mathbb{R})$ (see \cite{MR2668420}).
\end{itemize}
\end{example}

\medskip

To get some other example, we remark that $\mathscr{E}(\mathbb{R}^d)$ is clearly stable by multiplication by a trigonometric polynomial, derivation and the natural action of the affine group of $\mathbb{R}^d$. Furthermore, it enjoys the following tensor product property.
\begin{proposition}
If $d_1,d_2 \in \mathbb{N}^*$ then $\mathscr{E}(\mathbb{R}^{d_1})\otimes \mathscr{E}(\mathbb{R}^{d_2}) \subset \mathscr{E}(\mathbb{R}^{d_1+d_2})$.
\end{proposition}
\begin{example} $\partial_{v_1} e^{-v_1^4 - (v_1-3v_2)^2} \in \mathscr{E}(\mathbb{R}^2)$.
\end{example}
\subsection{An algebraic decomposition}
\label{Appendix2} The aim of this subsection is to prove Lemma \ref{lem_algb_dec}. 

\begin{proof}[Proof of Lemma \ref{lem_algb_dec}]
If $\omega=0$, we get the result by expanding the polynomial $(t-s)^n$. So we suppose now that $\omega\not=0$.
Since we recognize a convolution product, we apply a Laplace transform. So we get
\[ \Lap \left[ \int_0^t e^{i\omega s} s^m (t-s)^n \textrm{d}s \right] (z) = \Lap \left[ t^m \right] (z+\omega) \Lap \left[ t^n \right] (z) = \frac{n!m! (-i)^{n+m+2}}{(z+\omega)^{m+1}z^{n+1}}.   \]
We can apply a partial fraction decomposition to get some complex coefficients $(a_j)_{j=0,\dots,n}$ and $(b_j)_{j=0,\dots,m}$ such that
\[ \frac{n!m! (-i)^{n+m+2}}{(z+\omega)^{m+1}z^{n+1}} = \sum_{j=0}^{n} \frac{a_j}{z^{j+1}} + \sum_{j=0}^{m} \frac{b_j}{(z+\omega)^{j+1}}.    \]
Consequently, we have
\[ \Lap \left[ \int_0^t e^{i\omega s} s^m (t-s)^n \textrm{d}s \right] (z)  = \Lap \left[ \sum_{j=0}^{n} \frac{a_j i^{j+1}}{j!} t^j + \sum_{j=0}^{m} \frac{b_j i^{j+1}}{j!} t^j e^{i\omega t}  \right] (z), \]
Since the Laplace transform characterizes the continuous functions with an exponential order (see Theorem $1.7.3$ in \cite{MR2798103}), we have proved the lemma.
\end{proof}

\subsection{Computation of the zeros} \label{App:Computation of the zeros}
We have used the following Maple code to compute the zeros of $D_k$. We recall from Lemma \ref{Hilb_to_Four} that for $\Im(z)>0$
\begin{equation*}
\begin{split}
D_k(z) &= 1- \frac1{|k|^2} \int \frac{k\cdot \nabla_v f^{eq}(v)}{v\cdot k-z} \textrm{d}v 
= 1- \frac i{|k|^2}\Lap\left[ \Four [k\cdot\nabla_v f^{eq}(v)](kt) \right](z)\\
&=1- \frac i{|k|^2}\Lap\left[ ik\cdot(kt)\Four [f^{eq}(v)](kt) \right](z)=1+\Lap\left[ t\Four [f^{eq}(v)](kt) \right](z)\\
&=1+\int_0^\infty t\Four[f^{eq}(v)](kt)e^{izt}{\textrm{d}}t = 1+\int_0^\infty t\int_{\mathbb{R}^d}f^{eq}(v)e^{itk\cdot v}{\textrm{d}} v e^{izt}{\textrm{d}} t.
\end{split}
\end{equation*}
We can write $v=v_\parallel+v_\perp$, with $v_\parallel$ the component of $v$ along $k$ and  $v_\perp$ perpendicular to $k$,(when $d\ge 2$), so that
$$
D_k(z)  = 1+\int_0^\infty t\int_{(\mathbb{R}^d)_\parallel}\left(\int_{(\mathbb{R}^d)_\perp}f^{eq}(v_\parallel+v_\perp){\textrm{d}}v_\perp\right)e^{itk\cdot v_\parallel}{\textrm{d}} v_\parallel e^{izt}{\textrm{d}} t.
$$
\begin{remark}
\label{remDk}
Note that $D_{-k}(z)=D_k(z)$, if $\Four [f^{eq}(v)]$ is an even function.
\end{remark}

\begin{verbatim}
with(inttrans):
with(RootFinding):
Digits:=20:
feq:=exp(-((v)^2)/2);
#the space mode
k:=1.;
#Fourier transform of the equilibrium
Tfeq:=fourier(feq,v,t):
#the analytic function
Dk:=1+int(t*subs(t=k*t,Tfeq)*exp(I*om*t),t=0..infinity):
#the time modes
l:=sort([Analytic(Dk,om,re=-8..8,im=-8..8)],(a,b)->Im(a)>Im(b));
\end{verbatim}

\medskip
Received xxxx 20xx; revised xxxx 20xx.
\medskip

\end{document}